\documentclass[11pt, a4paper]{amsart}
\usepackage[margin=1in]{geometry}
\usepackage{tikz-cd}
\usepackage{amsmath,amssymb,amscd,amsthm}
\usepackage[T1]{fontenc}
\usepackage[utf8]{inputenc} 
\usepackage{mathrsfs}
\usepackage{mathtools}
\usepackage{hyperref}

\usepackage[all,cmtip]{xy}
\xyoption{all}
\input{xypic}

\renewcommand{\phi}{\varphi}

\renewcommand{\H}{\mathbb{H}}

\let\emptyset\varnothing

\newcommand{\C}{\mathbb{C}}

\newcommand{\N}{\mathbb{N}}

\newcommand{\R}{\mathbb{R}}

\newcommand{\Z}{\mathbb{Z}}

\newcommand{\RGC}{\mathcal{G}_\mathbb{C}^{\, \tau}}

\newcommand{\ov}{\overline}

\theoremstyle{plain}

\newtheorem{teorema}{Theorem}[section]
\newtheorem{prop}[teorema]{Proposition}
\newtheorem{corollary}[teorema]{Corollary}

\newtheorem{lemma}[teorema]{Lemma}

\theoremstyle{remark}

\newtheorem{remark}[teorema]{Remark}

\theoremstyle{definition}

\theoremstyle{plain}

\newtheorem{case}{Case}

\usepackage{tikz}
\usetikzlibrary{decorations.pathmorphing,arrows.meta}

\DeclareMathOperator{\Hom}{Hom}

\DeclareMathOperator{\de}{deg}
\DeclareMathOperator{\rk}{rk}

\DeclareMathOperator{\Gl}{GL}
\DeclareMathOperator{\End}{End}
\DeclareMathOperator{\Imm}{Im}

\DeclareMathOperator{\gr}{gr}

\DeclareMathOperator{\Gal}{Gal}

\DeclareMathOperator{\Pic}{Pic}
\DeclareMathOperator{\Spec}{Spec}
\DeclareMathOperator{\MD}{M_{Dol}}
\DeclareMathOperator{\MB}{M_{B}}

\renewcommand{\geq}{\geqslant}
\renewcommand{\leq}{\leqslant}

\title[Connected components of real loci]{Connected components of real loci in moduli spaces of vector and Higgs bundles over a Klein surface}

\author{Florent Schaffhauser}
\address{Universität Heidelberg, Institut f\"ur Mathematik, Im Neuenheimer Feld 205, 69120 Heidelberg}
\email{f.schaffhauser@mathi.uni-heidelberg.de}
\author{Tommaso Scognamiglio}
\address{Università di Bologna, Instituto Nazionale di Alta Matematica, Via Zamboni 33, 40126 Bologna}
\email{scognamiglio@altamatematica.it}

\counterwithin{equation}{section}

\begin{document}

\begin{abstract}
Let $X$ be a Riemann surface of genus $g \geqslant 2$ and let $\sigma : X \to X$ be an antiholomorphic involution on $X$. Let $\mathcal{N}(r,d)$ be the moduli space of semistable vector bundles of rank $r$ and degree $d$ on $X$, with the induced real structure. Using a gauge-theoretic approach, we determine the number of connected components of the real locus of $\mathcal{N}(r,d)$ for general $r$ and $d$. We show in particular that, when the base curve has real points, quaternionic vector bundles can exist for even rank and degree but that the number of connected components of $\R\mathcal{N}(r,d)$ is still equal to that of $\R\Pic_d$. In contrast, when the base curve has empty real locus and $r$ and $d$ are not coprime, the number of connected components of $\R\mathcal{N}(r,d)$ can be smaller than that of $\R\Pic_d$. We then generalize these results to real loci of moduli spaces of Higgs bundles and apply them to the study of the topology of certain $(A,A,A)$ and $(A,B,A)$ branes in the associated hyperkähler quotient.
\end{abstract}

\maketitle

\vspace{-1cm}

\tableofcontents

\vspace{-1cm}
\section{Introduction}

Let $X$ be a compact and connected Riemann surface of genus $g \geq 2$. We denote by $\mathcal{N}(r,d)$ the moduli space of (semistable) vector bundles of rank $r$ and degree $d$ over $X$. It is an irreducible projective variety of dimension $r^2(g-1)+1$, which is smooth whenever $r$ and $d$ are coprime. In general, the smooth points of $\mathcal{N}(r,d)$ are exactly the stable vector bundles (every vector bundle is stable if $\gcd(r,d)=1$). The  study of these varieties has attracted significant interest in algebraic and differential geometry since the 1970s, due to its connection with gauge theory, Morse theory, character varieties and more. We are interested in its study under the additional presence of a \textit{real structure} on $X$, i.e.\ an antiholomorphic involution $\sigma:X \to X$. In this case, the variety $\mathcal{N}(r,d)$ has an induced antiholomorphic involution $\sigma:\mathcal{N}(r,d) \to \mathcal{N}(r,d)$ given by $\sigma(E)=\sigma^*(\overline{E})$. The corresponding real points $\R \mathcal{N}(r,d)=\mathcal{N}(r,d)^{\sigma}$ form a real algebraic variety. The study of this real algebraic variety is closely related to the study of real and quaternionic vector bundles over $(X,\sigma)$. This setting was first studied in \cite{BHH}. Recall that a real/quaternionic vector bundle on $(X,\sigma)$ is a pair $(E,\tau)$, where $E$ is an holomorphic vector bundle on $X$ and $\tau:E \to E$ an antiholomorphic map covering $\sigma$ such that $\tau^2=Id$ (real) or $\tau^2=-Id$ (quaternionic). If $E \in \R \mathcal{N}^{st}(r,d)$ is a stable vector bundle, it can be endowed with an essentially unique either real or quaternionic structure and not both (see Proposition \ref{realHiggsbundlemoduli}). The open subset $\R \mathcal{N}^{st}(r,d)$ can thus be thought of as the moduli space of geometrically stable real and quaternionic stable vector bundles on $(\Sigma,\sigma)$. In particular, if $r,d$ are coprime, we have a modular description of $\R \mathcal{N}(r,d)$. For non-coprime $r$ and $d$, the strictly semistable points in $\R \mathcal{N}(r,d)$ do not admit such a modular description. We refer to Proposition \ref{polystable-real} for a more detailed analysis of what happens in this case. For $r=1$ (i.e.\ in the case of the Picard variety $\Pic_d$ of line bundles of degree $d$), the geometry of $\R \Pic_d$ has been known at least since the work of Gross and Harris in \cite{GH}. For $r \geq 2$, most of the information in the literature is concerned with the case when $r$ and $d$ are coprime, in which we have the modular interpretation of real points of $\mathcal{N}(r,d)$ mentioned above. The connected components of the geometrically stable stable locus of $\mathbb{R}\mathcal{N}(r,d)$, for instance, have been determined in \cite{Sch_JSG}, using that modular interpretation. The main result of that paper is that these connected components are indexed by the topological types of real and quaternionic vector bundles of rank $r$ and degree $d$ (see Section \ref{gauge_theory_moduli} for details). In the present paper, we determine the connected components of $\R \mathcal{N}(r,d)$ for all $r \geq 1$ and all $d \in \Z$ (Theorem \ref{main-result-intro}). Then, using the results of our previous paper \cite{Sch-Sco}, we obtain a description of the connected components of the real points of the moduli spaces of Higgs bundles (Theorem \ref{main-result-Higgs-intro}). Note that, for all $r$ and $d$, we can compose the determinant map $\det:\R \mathcal{N}(r,d) \to \R \Pic_d$ with the canonical map $\R\Pic_d \to \pi_0(\R \Pic_d)$ to obtain a continuous map $c_{r,d}:\R \mathcal{N}(r,d) \to \pi_0(\R \Pic_d)$. For every $s \in \pi_0(\R \Pic_d)$, we denote by $\R \mathcal{N}(r,d,s) \coloneqq c_{r,d}^{-1}(s)$ the fibre of that map above $s$. The main result of this article is the following (see Theorem \ref{mainresult}).

\begin{teorema}
\label{main-result-intro}
The fibres of the map $c_{r,d}$ are connected and the decomposition of $\R \mathcal{N}(r,d)$ into connected components is given as follows:
$$
\R \mathcal{N}(r,d)= \bigsqcup_{s \in \Imm c_{r,d}}\R \mathcal{N}(r,d,s)\ .
$$
\end{teorema}

As a corollary to Theorem \ref{main-result-intro}, we count the connected components of the real locus $\R\mathcal{N}(r,d)$ for all $r\geq 1$ and all $d\in\Z$ (see Table \ref{CC_count_curves_without_real_pts} and Theorem \ref{CC_count_curves_with_real_pts}). We will see that, when $\gcd(r,d) = 1$, the indexing set $\Imm\, c_{r,d}$ coincides with the set of topological types of real and quaternionic bundles of rank $r$ and degree $d$, in which case $\Imm\, c_{r,d} = \pi_0(\R\Pic_d)$. But, when $\gcd(r,d) \neq 1$ and $\R X = \emptyset$, the set $\Imm\, c_{r,d}$ might be strictly smaller than $\pi_0(\R \Pic_d)$. Namely, this happens when:
\begin{enumerate}
    \item $g$ is even, $d$ is odd and $r$ is even, in which case $\R\mathcal{N}(2r',2d'+1) = \emptyset$ even though $\R\Pic_{2d'+1} = \Pic_{2d'+1}^\H$ has one connected component.
    \item $g$ is odd, $d$ is even and $r$ is even, in which case the only non-empty fibre of $c_{r,d}$ is $\R\mathcal{N}(2r',2d') = c_{2r',2d'}^{-1}(\R)$ even though $\R\Pic_{2d'} = \Pic_{2d'}^\R \sqcup \Pic_{2d'}^\H$ has two connected components and $c_{2r',2d'}^{-1}(\R)$ contains quaternionic vector bundles.
\end{enumerate}
When $\R X \neq \emptyset$, the real locus $\R\mathcal{N}(r,d)$ always has $2^{n-1}$ connected components (Theorem \ref{CC_count_curves_with_real_pts}) and we show that, when $\gcd(r,d) \neq 1$, quaternionic vector bundles (which then necessarily have even rank and degree) are contained in the fibre $c_{r,d}^{-1}(0)$, meaning that the real part of the determinant bundle of a quaternionic vector bundle necessarily has trivial Stiefel-Whitney classes (see Lemma \ref{SW_class_for_quat_bundles}). The strategy to prove Theorem \ref{main-result-intro} consists in rewriting $\R \mathcal{N}(r,d,s)$ as a union of smaller subspaces and showing that these are connected and have a non-empty intersection. In this context, a fundamental step is to prove the connectedness of the set of $\R \mathcal{N}(r,d,s)$ given by the bundles that admit a real or quaternionic structure with a fixed topological type. This is indeed a relevant subspace to look at because the intersection $\R\mathcal{N}(r,d,s) \cap \mathcal{N}^{st}(r,d)$ consists entirely of such bundles. Note, however, that the strictly semistable locus of $\R\mathcal{N}(r,d,s)$ may, in addition to real and quaternionic bundles of fixed topological type, also contain bundles that admit both a real and a quaternionic structure, as well as bundles that admit neither a real nor a quaternionic structure. The study of the connectedness of the loci of $\R\mathcal{N}(r,d,s)$ consisting of real and quaternionic bundles of fixed topological type is achieved by gauge-theoretic methods. Namely, using gauge theory, we construct connected (non-empty) topological spaces whose points are in bijection with real or quaternionic semistable bundles on $(X,\sigma)$ of a given topological type. These spaces are mapped continuously to $\R \mathcal{N}(r,d)$, so the images are connected. This gauge-theoretic construction has already been used in \cite{Sch_JSG}, but note that these topological spaces are not moduli spaces in the strict sense of algebraic geometry: the points have the expected modular interpretation (as $S$-equivalence classes of semistable real and quaternionic vector bundles) but the algebraic structure of these moduli spaces is unclear. Since we are in the context of real algebraic geometry, a first natural question to ask is whether they carry a semi-algebraic structure. To state our next result, recall that a Higgs bundle over $X$ is a pair $(E,\phi)$, where $E$ is an holomorphic vector bundle over $X$ and $\phi:E \to E \otimes \Omega^1_X$, where $\Omega^1_X$ is the cotangent bundle. For all $r$ and $d$, let $\MD(r,d)$ be the moduli space of semistable Higgs bundles of rank $r$ and degree $d$ over $X$. This is an irreducible quasi-projective complex algebraic variety of dimension $2(r^2(g-1)+1)$. Viewing $\mathcal{N}(r,d)$ as the subset of $\MD(r,d)$ cut out by the equation $\varphi=0$, the map $c_{r,d}:\R \mathcal{N}(r,d) \to \pi_0(\R \Pic_d)$ can be extended to a map $\widehat{c_{r,d}}:\R \MD(r,d) \to \pi_0(\R \Pic_d)$ and we set $\R \MD(r,d,s) := \widehat{c_{r,d}}^{-1}(s)$. We then prove the following result (Theorem \ref{theoremHiggsbundles}).

\begin{teorema}\label{main-result-Higgs-intro}
For all $s \in \pi_0(\R \Pic_d)$, the subspace $\R \MD(r,d,s)$ is the connected component of $\R\MD(r,d)$ containing $\R\mathcal{N}(r,d,s)$.
\end{teorema}

Since $\Imm\, \widehat{c_{r,d}} = \Imm\, c_{r,d}$, we deduce in Corollary \ref{CC_branes} the number of connected components of $\R\MD(r,d)$. We refer to Section \ref{sectionà-Higgs} for the proof of Theorem \ref{main-result-Higgs-intro}. In the same paragraph, we also explain the consequences of Theorem \ref{main-result-Higgs-intro} on the topology of certain $(A,A,B)$ and $(A,B,A)$ branes in the moduli space of solutions to Hitchin's equations. More precisely, the real structure $\sigma : X \to X$ induces two conjugate real structures $\beta : (E,\phi) \mapsto (\sigma^*(\ov{E}), \sigma^*(\ov{\varphi}))$ and $\widetilde{\beta} : (E,\phi) \mapsto (\sigma^*(\ov{E}), -\sigma^*(\ov{\varphi}))$ which, on the Betti side of the nonabelian Hodge correspondence, correspond respectively to the real structure of the character variety $\MB(r,0) := \mathrm{Hom}(\pi_1 X, \Gl(r,\C)) /\negmedspace/ \Gl(r,\C)$ induced by $\sigma_* : \pi_1 X \to \pi_1 X$ and $g \mapsto \ov{g}$ on $\Gl(r,\C)$ (hence an $(A,A,B)$-brane $\mathrm{Fix}(\beta)$) and the algebraic involution of $\MB(r,0)$ induced by $\sigma_*$ and $g \mapsto \ ^t g^{-1}$ (hence an $(A,B,A)$-brane $\mathrm{Fix}(\widetilde{\beta})$). We then prove the following result in Theorem \ref{CC_branes2}.

\begin{teorema}
    The $(A,A,B)$ and $(A,B,A)$ branes $\mathrm{Fix}(\beta)$ and $\mathrm{Fix}(\widetilde{\beta})$ are homeomorphic. If $\R X = \emptyset$, the number of connected components of these fixed-point sets is given by Table \ref{CC_count_curves_without_real_pts}. If $\R X$ has $n > 0$ connected components, $\mathrm{Fix}(\beta)$ and $\mathrm{Fix}(\widetilde{\beta})$ have $2^{n-1}$ connected components.
\end{teorema}

\section{Moduli of semistable bundles and Galois structure}
\label{section-first}
\subsection{Moduli spaces of vector bundles on curves}
\label{notations}

Let $X$ be a compact and connected Riemann surface of genus $g \geq 2$. Recall that the slope $\mu(E)$ of a vector bundle $E$ on $X$ is defined as $\mu(E) \coloneqq \frac{\de(E)}{\rk(E)}$. We say that a vector bundle $E$ is semistable if, for every non-trivial  subbundle $F \subseteq E$, we have $\mu(F) \leq \mu(E)$. The vector bundle $E$ is said to be stable if the previous inequality is strict for every $F$. If $\gcd(r,d) = 1$, a vector bundle $E$ is semistable  if and only if it is stable. For a stable vector bundle $E$, there is an equality $\End(E)=\{\lambda Id\}_{\lambda \in \C}$. For every semistable  bundle $E$, there exists a filtration $0=E_{\ell} \subseteq E_{\ell-1} \subseteq \cdots \subseteq E_0=E$  such that, for each $j=1,\dots,l$, we have $\mu\left(E_{j-1}/E_j\right)=\mu(E)$ and $E_{j-1}/E_j$ is stable. We denote by $\gr(E)$ the associated graded object $\gr(E)\coloneqq E_{\ell} \oplus \cdots \oplus E/E_1.$ Its graded isomorphism class depends only on the semistable vector bundle $E$, not on the choice of the filtration satisfying the previous conditions \cite{Seshadri67}. The vector bundle $E$ is called polystable if it is semistable and $E \simeq \gr(E)$. Moreover, two semistable vector bundles $E$ and $E'$ are called $S$-equivalent if $\gr(E) \cong \gr(E')$. For all $r,d$, there exists a normal quasiprojective complex variety $\mathcal{N}(r,d)$  of dimension $r^2(g-1) + 1$ parameterizing $S$-equivalence classes of semistable vector bundles of rank $r$ and degree $d$ over $X$.  If $r,d$ are coprime, the moduli space $\mathcal{N}(r,d)$ is smooth, see for instance \cite{NS}. More generally, for all $r$ and $d$, there is an open smooth subvariety $\mathcal{N}^{st}(r,d) \subseteq \mathcal{N}(r,d)$, whose points are isomorphism classes of stable vector bundles of rank $r$ and degree $d$. Given $r,d$ as above, in what follows we set $m \coloneqq \gcd(r,d)$ and we put $r' \coloneqq r/m$ and $d' \coloneqq d/m$. Notice that $d'/r'=d/r$ and $\gcd(r',d')=1$. In particular, consider a polystable vector bundle $E \in \mathcal{N}(r,d)$ with $E \cong E_1 \oplus \cdots \oplus E_s$ where $E_1,\dots,E_s$ are stable vector bundles such that $\mu(E_i)=d/r$ for all $i$. Then there exist integers $m_1,\dots,m_s > 0$ such that $E_i \in \mathcal{N}^{st}(m_ir',m_id')$ for each $i=1,\dots,s$ (namely, $m_i := r' d_i = r_i d'$, where $d_i := \de(E_i)$ and $r_i := \rk(E_i)$).

\subsection{Real and quaternionic vector bundles}

Let $\sigma:X \to X$ be an antiholomorphic involution of $X$, also called a real structure. The involution $\sigma$ determines a real projective curve $X_{\R}$, i.e.\ a smooth projective variety of dimension $1$ over $\Spec(\R)$ such that $ X_{\R} \times_{\Spec(\R)} \Spec(\C) \simeq X$ and, via this isomorphism, $\sigma$ corresponds to the complex conjugation. The isomorphism above induces an identification $X_{\R}(\R)=X^{\sigma}$. We will denote by $n \geq 0$ the number of connected components of the fixed point set $X^{\sigma}$. Recall that each connected component of $X^{\sigma}$ is homeomorphic to a circle. In what follows, we will denote by $\Gamma\coloneqq \Gal(\C/\R)=\{1,\sigma\}$ the absolute Galois group of $\R$. We will also use the notation $\R X \coloneqq X^{\sigma}$. Given a holomorphic vector bundle $E \to X$, we will say that an antiholomorphic map $\tau: E \to E$ covers $\sigma$ if $\tau$ is fibrewise $\C$-antilinear and makes the following diagram commute.
$$
\begin{tikzcd}
E \ar[r,"\tau"'] \ar[d] &E \ar[d] \\
X \ar[r,"\sigma"'] &X
\end{tikzcd}
$$
Note that giving an antiholomorphic map $\tau:E\to E$ covering $\sigma$ is equivalent to giving a morphism of holomorphic vector bundles $\alpha:E \to \sigma^*(\overline{E})$. Moreover, $\tau^2 = \pm \mathrm{Id}_E$ if and only if $\sigma^*(\overline{\alpha}) \alpha = \pm \mathrm{Id}_E$. A real vector bundle is a pair $(E,\tau)$, where $E$ is a vector bundle and $\tau:E \to E$ is an antiholomorphic involution covering $\sigma:X \to X$.
Quaternionic vector bundles are defined similarly, except that $\tau^2 = - \mathrm{Id}_E$. A morphism of real/quaternionic vector bundles $f : (E_1,\tau_1) \to (E_2,\tau_2)$ is a morphism of vector bundles $E_1 \to E_2$ that commutes with $\tau_1$ and $\tau_2$. 

\begin{remark}
\label{remark-determinants}
If $(E,\tau)$ is a real vector bundle, the determinant line bundle $(\det(E),\det(\tau))$ is a real vector bundle. If $(E,\tau)$ is a quaternionic vector bundle, the determinant line bundle $(\det(E),\det(\tau))$ is real if $\rk(E)$ is even and quaternionic if $\rk(E)$ is odd. 
\end{remark}

If $(E,\tau)$ is a real vector bundle, the involution $\tau$ endows $E$ with a structure of real algebraic vector bundle over $X$, meaning that there exists an algebraic vector bundle $E_{\R}$ over $X_{\R}$ such that we have an isomorphism 
$E_{\R} \times_{\Spec(\R)} \Spec(\C) \cong E$
and, via this isomorphism, $\tau$ corresponds to the complex conjugation on the second factor. If $(E,\tau)$ is a real vector bundle whose underlying vector bundle $E$ is semistable (of rank $r$ and degree $d$, say), then the corresponding point $E \in \mathcal{N}(r,d)$ belongs to the $\Gamma$-fixed points. However, not all real points of $\R\mathcal{N}(r,d)$ come from real vector bundles. For instance, a quaternionic vector bundle whose underlying vector bundle is semistable also defines a real point.
\begin{remark}
If $\R X \neq \emptyset$ and $(E,\tau)$ is a quaternionic vector bundle over $X$, then $\rk(E)$ must be even. Indeed, taking a point $x \in \R X$, the map $\tau$ restricts to an antilinear map on the fiber $\tau_x :E_x \to E_x$ such that $\tau_x^2=-\mathrm{Id}$, i.e.\ to a quaternionic structure on the $\C$-vector space $E_x$, and having such a structure is only possible if the dimension of $E_x$ is even.
\end{remark}

More generally, we have the following criteria for the existence of real/quaternionic bundle on a Klein surface $(\Sigma,\sigma)$.

\begin{prop}{\cite[Theorem 2.4]{Sch_JSG}}
\label{propcriteriaexist}
We have the following conditions for the existence of real and quaternionic vector bundles:
\begin{enumerate}
    \item If $\R X= \emptyset$ then real vector bundles of rank $r$ and degree $d$ exist if and only if $d \equiv 0 \ \mathrm{mod}\ 2$. Quaternionic vector bundles exist if and only if $d+r(g-1) \equiv 0 \ \mathrm{mod}\ 2$.
    \item If $\R X\neq \emptyset$ then real vector bundles exist for any rank $r$ and degree $d$. Quaternionic vector bundles exist if and only if $r \equiv 0 \ \mathrm{mod}\ 2$ and $d\equiv 0 \ \mathrm{mod}\ 2$.
\end{enumerate}
\end{prop}

We also have the following result on the structure of the set of real points of the stable locus of $\mathcal{N}(r,d)$, which we denote by $\R\mathcal{N}^{st}(r,d)$. We refer to \cite[Theorem 3.8]{Sch_JSG} for further details on the structure of $\R\mathcal{N}^{st}(r,d)$, which will be called the \textit{geometrically stable} locus of $\R\mathcal{N}(r,d)$.

\begin{prop}{\cite[Proposition 2.8]{Sch_JSG}}
\label{realHiggsbundlemoduli}
Let $E$ be a stable vector bundle such that $\sigma^*(\ov{E}) \simeq E$, Then $E$ admits either a real or a quaternionic structure, but not both. Moreover, such a structure is unique up to isomorphism of real or quaternionic vector bundle. As a consequence, there is a decomposition of the geometrically stable locus of $\R\mathcal{N}(r,d)$ into a real and a quaternionic part:
\begin{equation}\label{decomp_real_quat_of_geom_stable_locus}
\R\mathcal{N}^{st}(r,d) =: \mathcal{N}^{st,\R}(r,d) \sqcup \mathcal{N}^{st,\H}(r,d) 
\end{equation}
\end{prop}

\begin{remark}
\label{remarkrealstability}
A real/quaternionic vector bundle $(E,\tau)$ is called (semi)stable if for each $\tau$-invariant subbundle $F \subseteq E$, we have $\mu(F) (\leq) \mu(E)$. If a real/quaternionic vector bundle $(E,\tau)$ is real/quaternionic stable, it is not always the case that the underlying vector bundle $E$ is stable (when it is, we will call $(E,\tau)$ \textit{geometrically stable}). Indeed, consider for instance a stable vector bundle $\mathcal{F}$ such that $\mathcal{F} \neq \sigma^*(\ov{\mathcal{F}})$. Then the vector bundle $E=\mathcal{F} \oplus \sigma^*(\ov{\mathcal{F}})$ is polystable but not stable and can be endowed with a real structure $\tau_1(x,y)=(\widetilde{\sigma}^{-1}(y)\widetilde{\sigma}(x))$ and a quaternionic structure $\tau_2(x,y)=(\widetilde{\sigma}(y),-\widetilde{\sigma}(x))$, where $\widetilde{\sigma} : \sigma^*(\ov{F}) \to F$ is the canonical map covering $\sigma : X \to X$. It is not hard to check that both $(E,\tau_1)$ and $(E,\tau_2)$ are real/quaternionic stable.   
\end{remark}

We then have the following result, classifying real/quaternionic semistable and stable bundles.

\begin{prop}{\cite[Proposition 2.7]{Sch_JSG}}
\label{realquaternionicstability}
A real/quaternionic bundle $(E,\tau)$ is real/quaternionic semistable if and only if the underlying bundle $E$ is semistable.  If $(E,\tau)$ is stable in the real/quaternionic sense, we have the following possibilities:
\begin{itemize}
    \item The underlying vector bundle $E$ is stable, in which case we will say that $(E,\tau)$ is \textit{geometrically stable}.
    \item The underlying vector bundle $E$ is not stable, in which case there exists a stable  bundle $\mathcal{E}$ such that $\mathcal{E} \not\simeq (\sigma^*\overline{\mathcal{E}})$ and an isomorphism \begin{equation}
    \label{notgeomstable}
        E \simeq \mathcal{E} \oplus \sigma^*(\overline{\mathcal{E}})
    \end{equation} with real/quaternionic structure $\tau$ given by exchanging the factors in (\ref{notgeomstable}) (see Remark \ref{remarkrealstability}).
\end{itemize}
\end{prop}

If $r,d$ are not coprime, then not all real points $E \in \R\mathcal{N}(r,d)$ admit the modular interpretation explained above. We can, however, still give the following description of polystable bundles in the real locus, as discussed in \cite[Section 2.5]{Sch_JSG}.

\begin{prop}
\label{polystable-real}
Given a polystable vector bundle $E\in \R \mathcal{N}(r,d)$, there  exist stable real or quaternionic bundles $(E_1,\tau_1),$ $\dots,(E_\ell,\tau_\ell)$, of slope $\frac{d}{r}$ and an isomorphism of vector bundles $$E \simeq E_1 \oplus \cdots \oplus E_\ell .$$ \end{prop}

Combining Propositions \ref{realquaternionicstability} and \ref{polystable-real}, we see that, for every polystable vector bundle $E\in \R \mathcal{N}(r,d)$, there exist vector bundles $\mathcal{E}_i \in \R \mathcal{N}^{st}(s_ir',s_id')$ and $\mathcal{F}_i \in \mathcal{N}^{st}(h_ir',h_id')$, and natural numbers $s_1,\dots,s_k$ and $h_1,\dots,h_l$ such that 
$$E \cong \mathcal{E}_1 \oplus \cdots \oplus \mathcal{E}_k \oplus \big(\mathcal{F}_1 \oplus \sigma^*(\overline{\mathcal{F}}_1)\big) \oplus \cdots \oplus \big(\mathcal{F}_l \oplus \sigma^*(\overline{\mathcal{F}}_l)\big).$$

\subsection{Real points of Picard varieties}
\label{realPicard}
Notice that every vector bundle of rank $1$ is semistable and we have thus an equality $\mathcal{N}(r,d)=\Pic_d$, where $\Pic_d$ is the connected component of the Picard group of $X$ parameterizing line bundles of degree $d$ on $X$. As in \eqref{decomp_real_quat_of_geom_stable_locus}, we write $\Pic_d^{\R}$ for line bundles $L \in \R\Pic_d$ admitting a real structure and $\Pic_d^{\H}$ for line bundles $L \in \R\Pic_d$ admitting a quaternionic structure. We now review the description of the real locus $\R \Pic_d$ for any $d$, obtained in \cite{GH}.

\subsubsection{Case when the curve has real points}\label{picardwithrealpoints} 

Assume first that $\R X \neq \emptyset$, with $n > 0$ connected components. Then $\R\Pic_d = \Pic_d^\R$ consists of isomorphism classes of real line bundles and it has $2^{n-1}$ connected components (\cite{GH}). These connected components are distinguished by the first Stiefel-Whitney class $s = (s_1, \ldots, s_n) \coloneqq w_1(\R L) \in H^1(\R X; \Z/2\Z) \simeq (\Z / 2\Z)^n$ of a real line bundle $L \in Pic_d^\R$, which is subject to the condition $s_1 + \ldots + s_n = d\ \mathrm{mod}\,{2}$. More precisely, the continuous map
$$
\begin{array}{rcl}
    \R\Pic_d & \longrightarrow & (\Z/2\Z)^n  \\
    L & \longmapsto & s \coloneqq w_1(\R L)
\end{array}
$$ induces a bijection 
$
\pi_0(\R\Pic_d) 
\simeq 
\{ s \in (\Z/2\Z)^n\ |\ s_1 +\, \ldots\, + s_n = d\ \text{mod}\ 2\} \simeq (\Z/2\Z)^{n-1}.
$ 

\subsubsection{Case when the curve has no real points}\label{picardwithoutrealpoints}
Assume now that $\R X = \emptyset$. When $X$ has no real points, real vector bundles must have even degree, while quaternionic vector bundles can now occur in odd rank (subject to the condition $d + r(g-1) \equiv 0\ \mathrm{mod}\,2$). The real locus of the Picard variety $\Pic_d$ is given as follows \cite{GH}:
\begin{itemize}
\item $\R\Pic_{2d} = \Pic^{\R}_{2d}$ if $g$ is even.
\item $\R\Pic_{2d} = \Pic^{\R}_{2d} \sqcup \Pic^{\H}_{2d'}$ if $g$ is odd.
\item $\R\Pic_{2d+1} = \Pic^{\H}_{2d+1}$ if $g$ is even.
\item $\R\Pic_{2d+1} = \emptyset$ if $g$ is odd.
\end{itemize}
So, when $\R\Pic_d\neq\emptyset$, the pointed set $\pi_0(\R\Pic_d)$ consists of either one or two points. We denote the elements of $\pi_0(\R\Pic_d)$ by $\R$ and $\H$. Namely, $\R$ denotes the connected component of $\R\Pic_d$ containing real line bundles (if there are any), and $\H$ denotes the connected component of $\R\Pic_d$ containing quaternionic line bundles (if there are any).

\section{Gauge-theoretic construction of moduli spaces}\label{gauge_theory_moduli}

The idea that moduli of real objects should just be real points of moduli spaces defined over the real numbers is very natural. Unfortunately, the presence of non-trivial automorphisms is a well-known obstruction to that, the classic distinction being that the field of moduli is in general smaller than the field of definition. This is precisely what happens for vector bundles: due to the presence of non-trivial automorphisms (even for a stable vector bundle), a real point of the moduli space is not necessarily defined over the reals (it could be a quaternionic vector bundle, or a direct sum of a real and a quaternionic one, \textit{etc}). An additional complication, also due to the presence of non-trivial automorphisms is that objects defined over $\mathbb{R}$ which are isomorphic over $\mathbb{C}$ are not necessarily isomorphic over $\mathbb{R}$, or even $S$-equivalent in the case of vector bundles (see \cite{Sch_JSG}). So, to construct topological spaces that behave as moduli spaces of real (and quaternionic) vector bundles in a reasonable way, one has to use another approach. In the present context, the one that makes sense is the gauge-theoretic approach \cite{BHH, Sch_GD}. This kind of gauge theoretic approach was later extended to the case of any reductive group $G$ and principal $G$-bundles in \cite{BGH}.

\subsection{Lagrangian quotients}

The idea in \cite{Sch_GD} is to look at Lagrangian quotients of the space of unitary connections. More precisely, if we fix a $C^\infty$ real vector bundle $(E,\tau)$ over a compact Klein surface $(X,\sigma)$, equipped with a Hermitian metric $h$, the space of unitary connection $\mathcal{A}_h$ on that bundle is equipped with an induced anti-symplectic involution $\beta_\tau : A \mapsto A^\tau$, whose fixed points are in bijection with $\tau$-compatible holomorphic structures on $E$, meaning those holomorphic structures that make $\tau$ anti-holomorphic. Moreover, any two such holomorphic structures are isomorphic over $\R$ if and only if the corresponding connections are contained in a same orbit of the real gauge group $\RGC$, meaning the group of automorphisms of $E$ that commute with $\tau$ (see \cite{Sch_GD}). From that point of view, the set of isomorphism classes of $\tau$-compatible holomorphic structures on $(E,\tau)$ is in bijection with $\mathcal{A}_h^\tau / \RGC$, where $\mathcal{A}_h^\tau := \mathrm{Fix}(\beta_\tau)$. As a matter of fact, an analogous result holds for quaternionic vector bundles as well: if $(E,\tau)$ is quaternionic, the induced transformation $\beta_\tau : \mathcal{A}_h \to \mathcal{A}_h$ is still an involution (because the center of the gauge group acts trivially on $\mathcal{A}_h$), and it characterizes $\tau$-compatible holomorphic structures on $E$, as noted in \cite{Sch_JSG}. So it is easy to construct a topological space of real or quaternionic vector bundles that are smoothly isomorphic to a given $(E,\tau)$, but in order to obtain moduli spaces that are finite dimensional manifolds, we need to restrict to unitary connections defining \textit{polystable} $\tau$-compatible holomorphic structures on $E$. To achieve that, we adapt Donaldson's approach and study Yang-Mills connections with a Galois symmetry. Recall that a unitary connection $A$ is called a Yang-Mills connection if its curvature $F_A$, seen as a section of the bundle of Hermitian endomorphisms of $E$ is constant and central, meaning that, necessarily, $F_A = \frac{d}{r} \mathrm{Id}_E$, where $d$ is the degree of $E$. The set $\mathcal{A}_{YM} := F^{-1}(\frac{d}{r} \mathrm{Id}_E)$ of Yang-Mills connections is invariant under the involution $\beta_\tau$ and the action of the real part $\mathcal{G}_h^\tau := \mathcal{G}_h \cap \RGC$ of the unitary gauge group $\mathcal{G}_h$ and the following analogue of Donaldson's theorem holds.

\begin{teorema}{\cite[Theorem 3.7]{Sch_JSG}}\label{lag_quot} 
    Let $(E,h,\tau)$ be a $C^\infty$ Hermitian vector bundle equipped with a real (resp.\ quaternionic) structure. There is a bijection between the set of $S$-equivalence classes of real (resp.\ quaternionic) vector bundles that are smoothly isomorphic to $(E,\tau)$ and the Lagrangian quotient $\mathcal{A}_{YM}^{\ \tau} / \mathcal{G}_h^\tau := \big(F^{-1}(\textstyle\frac{d}{r} \mathrm{Id}_E) \cap \mathcal{A}_h^\tau\big) \big/ \mathcal{G}_h^\tau\ .$
\end{teorema}

This means that, given a real/quaternionic Hermitian vector bundle $(E,\tau)$ on $(X,\sigma)$, a $\beta_\tau$-fixed unitary connection $A \in \mathcal{A}_h^\tau$ defines a polystable $\tau$-compatible holomorphic structure on $(E, \tau)$ if and only if $A$ is a Yang-Mills connection, meaning that $A \in \mathcal{A}_{YM}^{\ \tau} := \mathcal{A}_h^\tau \cap \mathcal{A}_{YM}$. Moreover, any two such connections define isomorphic real or quaternionic structures if and only if they are contained in a same orbit of the action of the real part of the unitary gauge group $\mathcal{G}_h^\tau$. In this statement, the polystability condition is to be understood in the real or quaternionic sense, and part of the proof is showing that this is equivalent to being polystable as a complex object \cite{Sch_JSG}. As mentioned before, this property does not hold for stable real or quaternionic vector bundles, which in general are only polystable when viewed as holomorphic vector bundles. However, by looking at Galois-invariant Yang-Mills connnections that are irreducible as unitary connections (meaning that their stabilizer is the center of the unitary gauge group), we can construct moduli spaces of geometrically stable real/quaternionic vector bundles in a similar fashion. We refer to \cite{Sch_GD} for the details of that construction.

\subsection{Connectedness}

The Lagrangian quotients of Theorem \ref{lag_quot} come equipped with maps
\begin{equation}\label{map_from_Lag_quot}
\begin{array}{rcl}
f_\tau : \mathcal{A}_{YM}^{\ \tau} / \mathcal{G}_h^\tau & \longmapsto & \mathcal{A}_{YM} / \mathcal{G}_h \\
 \mathcal{G}_h^\tau \cdot A & \longmapsto &  \mathcal{G}_h \cdot A
\end{array}
\end{equation}
sending the $\mathcal{G}_h^\tau$-orbit of a Galois-invariant connection $A \in \mathcal{A}_{YM}^{\ \tau} \subset \mathcal{A}_{YM}$ to its $\mathcal{G}_h$-orbit. Note that $\tau$ induces an involution on $\mathcal{A}_{YM} / \mathcal{G}_h$ that in fact depends only on the real structure $\sigma : X \to X$. Indeed, by the Donaldson-Narasimhan-Seshadri theorem, there is a homeomorphism $\mathcal{A}_{YM} / \mathcal{G}_h \simeq \mathcal{N}(r,d)$ and the induced involution is just $\mathcal{E} \mapsto \sigma^*\overline{\mathcal{E}}$. The point is that there is an inclusion
$
\Imm\, f_\tau \subset \mathrm{Fix}_\sigma \big( \mathcal{A}_{YM} / \mathcal{G}_h \big) = \mathbb{R} \mathcal{N}(r,d)\ ,
$
meaning that Lagrangian quotients of (Galois-invariant) Yang-Mills connections are mapped into real points of the moduli space of semistable holomorphic vector bundles of rank $r$ and degree $d$. Moreover, in the stable locus of $\mathcal{N}(r,d)$, any real point comes from a uniquely determined connection $A \in \mathcal{A}_{YM}^{\ \tau}$ on some real or quaternionic vector bundle $(E,\tau)$, because a Galois-invariant geometrically stable vector bundle $\mathcal{E} \simeq \sigma^* \overline{\mathcal{E}}$ is either real or quaternionic, in a unique way up to real or quaternionic isomorphism (see Proposition \ref{realHiggsbundlemoduli}). So we can write the geometrically stable locus $\mathbb{R} \mathcal{N}^{st}(r,d)$ as a disjoint union of restricted Lagrangian quotients:
\begin{equation}\label{restricted_Lag_quot}
\mathbb{R} \mathcal{N}^{st}(r,d) \simeq 
\bigsqcup_{[\tau] \in I_{r,d}} \mathcal{A}_{YM}^{\ \tau, st} / \mathcal{G}_h^\tau\ ,
\end{equation}
where $\mathcal{A}_{YM}^{\ \tau, st} := f_\tau^{-1}(\Imm\, f_\tau \cap \mathcal{N}^{st}(r,d))$ and $I_{r,d}$ is the set of so-called \text{topological types} of real and quaternionic vector bundles of rank $r$ and degree $d$ over $(X, \sigma)$, i.e.\ the set of equivalence classes of such objects up to $C^\infty$ isomorphism. This set is finite because, as shown in \cite{BHH}, real or quaternionic vector bundles are topologically isomorphic if and only if they have the same rank and degree, as well as the same real invariants $w_1(E^\tau)$ when $\tau$ is real and $X^\sigma \neq \emptyset$. In particular, when $\gcd(r,d) = 1$, we have $\mathcal{N}^{st}(r,d) = \mathcal{N}(r,d)$ and
\begin{equation}\label{decomp_real_locus_coprime_case}
\mathbb{R} \mathcal{N}(r,d) \simeq 
\bigsqcup_{[\tau] \in I_{r,d}} \mathcal{A}_{YM}^{\ \tau} / \mathcal{G}_h^\tau\ ,
\end{equation}
the upshot being that, in that case, this is exactly the decomposition in connected components of $\mathbb{R} \mathcal{N}(r,d)$, as follows from Theorem \ref{Lag_quot_connected}. Note, however, that this result holds without assuming that $r$ and $d$ are coprime.

\begin{teorema}\label{Lag_quot_connected}
    Let $g \geq 2$, $r \geq 1$ and $d \in \Z$ be integers. Let $(X,\sigma)$ be a compact connected Klein surface of genus $g$ and let $(E,h)$ be a $C^\infty$ Hermitian vector bundle of rank $r$ and degree $d$, equipped with a real or quaternionic structure $\tau$. Let $\mathcal{A}_{YM}^{\ \tau}$ be the space of Galois-invariant Yang-Mills connections on $(E, h, \tau)$ and let $\mathcal{G}_h^\tau$ be the real part of the $h$-unitary gauge group of $E$. Then the Lagrangian quotient $\mathcal{A}_{YM}^{\ \tau} / \mathcal{G}_h^\tau$ is non-empty and connected.
\end{teorema}

\begin{proof}
    By \cite[Cor~3.17]{Sch_JDG}, the space $\mathcal{A}_{YM}^{\ \tau}$ is a deformation retract of $\mathcal{A}_{\mu_{ss}}^{\ \tau}$ (the space of $\tau$-compatible semistable holomorphic structures on $(E,\tau)$), so it suffices to show that $\mathcal{A}_{\mu_{ss}}^{\ \tau}$ is non-empty and connected. The proof goes by induction on $r$. When $r=1$, every holomorphic line bundle is stable, so $\mathcal{A}_{YM}^{\ \tau} = \mathcal{A}_h$, which is an affine space so it is in particular non-empty and connected. Take now $r \geq 2$. The complement of $\mathcal{A}_{\mu_{ss}}^{\ \tau}$ in $\mathcal{A}_h$ is the union of the non-semistable Harder-Narasimhan strata of $\mathcal{A}_h$, meaning that
    $$
    \mathcal{A}_h \setminus \mathcal{A}_{\mu_{ss}}^{\ \tau}
    = \bigsqcup_{\mu \not= \mu_{ss}} \mathcal{A}_\mu^\tau\ ,
    $$
    where $\mathcal{A}_\mu^\tau = \mathcal{A}_\mu \cap \mathcal{A}^\tau$ coincides with the Galois-invariant part of the complex Harder-Narasimhan stratum $\mathcal{A}_\mu$ \cite{Sch_JSG}. In particular, each $\mathcal{A}_\mu^\tau$ has finite (real) codimension in $\mathcal{A}_h^\tau$, given by 
    $$
    d_\mu := \sum_{1 \leq i < j \leq \ell} \big(d_i r_j - d_j r_i + r_i r_j (g-1)\big)
    $$
    where $\ell$ is the length of the Harder-Narasimhan filtration $0 = \mathcal{E}_0 \subset \mathcal{E}_1 \subset \ldots \subset \mathcal{E}_\ell = \mathcal{E}$ of any $\mathcal{E} \in \mathcal{A}_\mu$ and $(r_i, d_i)_{1 \leq i \leq \ell}$ are the rank and degree of the semistable quotients $\mathcal{E}_i / \mathcal{E}_{i-1}$, satisfying $\frac{d_1}{r_1} > \ldots > \frac{d_\ell}{r_\ell}$. In particular, the strata $\mathcal{A}_\mu$ are indexed by the Harder-Narasimhan types $\mu = (r_i, d_i)_{1 \leq i \leq \ell}$ of holomorphic vector bundles of rank $r$ and degree $d$ on $X$ and $\mathcal{A}_{\mu_{ss}}$ is the only open stratum (case when $\ell = 1$). Note that, for given $\mu$ and $\tau$, the stratum $\mathcal{A}_\mu^{\tau}$ may or not be empty, and we refer to \cite{LS} for further details on this. At any rate, for all $\mu \not= \mu_{ss}$, the stratum $\mathcal{A}_{\mu}^{\tau}$ is contained in the closure of a stratum $\mathcal{A}_{\mu'}^{\tau}$ for some $\mu'$ of the form $((r_1, d_1), (r_2, d_2))$. Since $r \geq 2$ and $g \geq 2$, one has
    $
    d_{\mu'} = d_1 r_2 - d_2 r_1 + r_1 r_2 (g - 1) \geq 1 + (r-1)(g-1) \geq 2
    $
    so the complement of $\mathcal{A}_{\mu_{ss}}^{\ \tau}$ in $\mathcal{A}_h$ is a countable union of closed submanifolds of codimension greater than $2$, which proves that $\mathcal{A}_{\mu_{ss}}^{\ \tau}$ is connected and dense in the affine space $\mathcal{A}_h$. In particular, it is non-empty.
\end{proof}

Thus, Lagrangian quotients of Galois-invariant Yang-Mills connections are connected and this fact will be used in our proof of Theorem \ref{mainresult}, giving the connected components of $\mathbb{R} \mathcal{N}(r,d)$ in the general case (i.e.\ without assuming that $\gcd(r,d) = 1$). 

\begin{remark}
    Using results of Daskalopoulos and Uhlenbeck, Theorem \ref{Lag_quot_connected} can be strengthened to the geometrically stable case, meaning that the restricted Lagrangian quotients $\mathcal{A}_{YM}^{\ \tau, st} := f_\tau^{-1}(\Imm\, f_\tau \cap \mathcal{N}^{st}(r,d))$ are connected, provided $g > 2$ or $r >2$. The case $g=2$ and $r=2$ can be deduced from the work of Narasimhan and Ramanan in \cite{NR} (see \cite{Wang}). This implies that Decomposition \eqref{restricted_Lag_quot} is the decomposition into connected components of the geometrically stable locus $\R\mathcal{N}^{st}(r,d)$, for all $r$ and $d$, so these connected components are indexed by the topological types $I_{r,d}$ of real and quaternionic vector bundles.
\end{remark}

\subsection{The determinant morphism}\label{open_closed_decomp}

For arbitrary $r$ and $d$, let us define a morphism of real algebraic varieties 
\begin{equation}\label{obstruction_map}
    c_{r,d} : \R\mathcal{N}(r,d) \to \pi_0(\R\Pic_d)\, ,
\end{equation} by composing the determinant morphism $\det : \R\mathcal{N}(r,d) \to\R\Pic_d$
with the canonical continuous map $\R\Pic_d \to \pi_0(\R\Pic_d)$, projecting a real point of $\Pic_d$ to the connected component of $\R\Pic_d$ that contains it. For every $s \in \pi_0(\R\Pic_d)$, we put 
$
\R\mathcal{N}(r,d,s)\coloneqq c_{r,d}^{-1}(s)\ .
$
Since $\pi_0(\R\Pic_d)$ is discrete, this gives us an open-closed decomposition
\begin{equation}
\label{fibers-topologicaltype}
\R\mathcal{N}(r,d) \quad = \bigsqcup_{s\in\pi_0(\R\Pic_d)} \R\mathcal{N}(r,d,s)\,.
\end{equation}
We will see in Theorem \ref{mainresult} that, for all $s\in\pi_0(\R\Pic_d)$, the space $\R\mathcal{N}(r,d,s)$ in the open-closed decomposition \eqref{fibers-topologicaltype} is connected. Thus, to obtain the decomposition of $\R \mathcal{N}(r,d)$ into connected components, it suffices to know which ones are non-empty, which we will also determine. For now, we restrict our attention to the case when $r$ and $d$ are coprime, as this will be the base case of our proof by induction on $m := \gcd(r,d)$ in the next section.

\begin{teorema}\label{CC_coprime_case}
    Let $r$ and $d$ be integers. If $\gcd(r, d) = 1$, then, for all $s \in \pi_0(\R\Pic_d)$, the fibre $\R\mathcal{N}(r,d,s) := c_{r,d}^{-1}(s)$ is non-empty. As a consequence, when $\gcd(r,d)=1$, the determinant morphism $E \mapsto \det E$ induces a bijection $\pi_0\big(\R\mathcal{N}(r,d)\big) \simeq \pi_0\big(\R\Pic_d\big).$
\end{teorema}

\begin{proof}
    Since $\gcd(r,d) = 1$, Decomposition \eqref{decomp_real_locus_coprime_case} holds. By Theorem \ref{Lag_quot_connected}, each Lagrangian quotient in this union is non-empty and connected. So the connected components of $\R\mathcal{N}(r,d)$ are indexed by the topological types of real and quaternionic vector bundles, which, by the results of \cite{BHH}, are given as follows:
    \begin{itemize}
        \item If $\R X \not= \emptyset$, real vector bundles $(E_1,\tau_1)$ and $(E_2,\tau_2)$ are topologically isomorphic if and only if they have the same rank and degree and the vector bundles $E_1^{\tau_1}$ and $E_2^{\tau_2}$ over $\R X$ have the same first Stiefel-Whitney class $w_1(E_i^{\tau_i}) = (s_1, \dots , s_n)$, where $n$ is the number of connected components of $\R X$. Similarly, quaternionic vector bundles $(E_1,\tau_1)$ and $(E_2,\tau_2)$ are topologically isomorphic if and only if they have the same rank and degree.
        \item If $\R X = \emptyset$, then real or quaternionic vector bundles are topologically isomorphic if and only if they have the same rank and degree.
    \end{itemize}
    There only remains to check that the determinant morphism induces a bijection $\pi_0(\R\mathcal{N}(r,d)) \simeq \pi_0(\R\Pic_d)$. Assume first that $\R X \neq \emptyset$. Since $\gcd(r,d) = 1$ by assumption, there can be no quaternionic vector bundles in this case (see Proposition \ref{propcriteriaexist}). For the same reason, there are no quaternionic line bundles in this case. So, comparing the topological classification of real vector bundles recalled above with the results of Gross and Harris recalled in Section \ref{picardwithrealpoints}, and using the fact that $\deg(E) := \deg(\det E)$ and $w_1(E^\tau) := w_1(\det E^\tau)$, we get indeed a bijection $\pi_0(\R\mathcal{N}(r,d)) \simeq \pi_0(\R\Pic_d)$.
    
    Assume now that $\R X = \emptyset$. Then, again by Proposition \ref{propcriteriaexist}, real vector bundles must satisfy $d \equiv 0\ \mathrm{mod}\ 2$, while quaternionic vector bundles must satisfy $d + r(g-1) \equiv 0\ \mathrm{mod}\ 2$. A case by case comparison with the results of Gross and Harris recalled in Section \ref{picardwithoutrealpoints} then shows that we have again a bijection $\pi_0(\R\mathcal{N}(r,d)) \simeq \pi_0(\R\Pic_d)$. In particular, when $\R X = \emptyset$ and $\gcd(r,d) = 1$, we have:
    \begin{itemize}
        \item $\R\mathcal{N}(2r+1,2d) = \mathcal{N}^{st,\R}(2r+1,2d)$ if $g$ is even.
        \item $\R\mathcal{N}(2r+1,2d) = \mathcal{N}^{st,\R}(2r+1,2d) \sqcup \mathcal{N}^{st,\H}(2r+1,2d)$ if $g$ is odd.
        \item $\R\mathcal{N}(2r+1,2d+1) = \mathcal{N}^{st,\H}(2r'+1,2d'+1)$ if $g$ is even.
        \item $\R\mathcal{N}(2r+1,2d+1) = \emptyset$ if $g$ is odd.
        \item $\R\mathcal{N}(2r, 2d+1) = \emptyset$.
    \end{itemize}
    This is consistent with the results of \cite[Theorem~1.1]{Sch_JSG}.
\end{proof}

Note that, given a real or quaternionic structure $\tau$ on a smooth Hermitian vector bundle $(E,h)$, the image of the map $f_\tau$ defined in \eqref{map_from_Lag_quot} is contained in $\mathbb{R}\mathcal{N}(r,d,s_\tau) = c_{r,d}^{-1}(s_\tau)$ where $s_\tau \in \pi_0(\mathbb{R}\mathrm{Pic}_d)$ is the connected component of $\mathbb{R}\mathrm{Pic}_d$ determined by the smooth Galois-invariant line bundle $(\det(E),\det(\tau))$, as follows from the definition of the map $c_{r,d} : \mathbb{R}\mathcal{N}(r,d) \to \mathbb{R} \mathrm{Pic}_d$ in \eqref{obstruction_map}. Thus, using the definition $\mathbb{R}\mathcal{N}(r,d,s) := c_{r,d}^{-1}(s)$ for all $s \in \pi_0(\R \mathrm{Pic}_d)$, we get the following commutative diagram, for all real or quaternionic structure $\tau$ on $(E,h)$. 
$$
\begin{tikzcd}
    \mathcal{A}_{YM}^{\ \tau} / \mathcal{G}_h^\tau \ar[r, "f_\tau"] \ar[rd] & \mathbb{R}\mathcal{N}(r,d) = \bigsqcup_{s \in \pi_0(\mathbb{R} \mathrm{Pic}_d) } \mathbb{R} \mathcal{N}(r,d,s) \\
    & \mathbb{R}\mathcal{N}(r,d,s_\tau) \ar[u, hook]
\end{tikzcd}
$$
The point is that it is not true in general that $s_\tau$ determines the topological type of $(E,\tau)$ as real or quaternionic vector bundle. More precisely, this holds in the coprime case but not in general. If for instance $(E,\tau)$ is a quaternionic vector bundle of rank $2r$, then $(\det E, \det\tau)$ will be a \textit{real} line bundle. So although we will prove that the open-closed decomposition \eqref{fibers-topologicaltype} determines the connected components of $\R\mathcal{N}(r,d)$, it might happen that $\R\mathcal{N}(r,d,s) = \emptyset$ for some particular $s \in \pi_0(\R\Pic_d)$, making the number of connected components $\R\mathcal{N}(r,d,s)$ smaller than $|\pi_0(\R\Pic_d)|$. We will for instance see that, when $\R X = \emptyset$ and $g$ is odd, then $\R\mathcal{N}(2r,2d)$ is connected (containing both real and quaternionic vector bundles), while $\R\Pic_{2d}$ has two connected components (one consisting of real line bundles and one consisting of quaternionic line bundles). This is consistent because, by the previous remark, $\H \not\in\Imm\, c_{2r,2d}$ in this case. We refer to Table \ref{CC_count_curves_without_real_pts} to see that the generalisation of Theorem \ref{CC_coprime_case} does not hold when $\R X = \emptyset$, and to Theorem \ref{CC_count_curves_with_real_pts} to see that it does hold when $\R X \not= \emptyset$.

\section{Connected components of the real locus}\label{topology_real_locus}

For all $r$ and $d$ and all $s \in \pi_0(\R \Pic_d X)$, consider the fibre $\R \mathcal{N}(r,d,s) := c_{r,d}^{-1}(s)$ of the map $c_{r,d}$ introduced in \eqref{obstruction_map}. In this section, we show that the fibre of $c_{r,d}$ are connected and prove our first main result below (see Theorem \ref{main-result-intro}).

\begin{teorema}
\label{mainresult}
For all integers $r\geq 1$ and $d\in \Z$, let $c_{r,d} : \R\mathcal{N}(r,d) \to \pi_0(\R\Pic_d)$ be the map defined in \eqref{obstruction_map}. Then, for all $s \in \pi_0(\R \Pic_d)$, the space $\R \mathcal{N}(r,d,s)$ is connected, so the decomposition of $\R \mathcal{N}(r,d)$ into connected components is given as follows:
$$
\R \mathcal{N}(r,d)= \bigsqcup_{s \in \Imm c_{r,d}}\R \mathcal{N}(r,d,s)
$$
\end{teorema}

\begin{proof}

The proof goes by induction on $m := \gcd(r,d)$ and is divided into several cases. The general strategy consists in writing the real loci $\R \mathcal{N}(r,d,s)$ as the union of certain smaller connected sets and showing that these smaller connected sets are pairwise non-disjoint. The base case $m = 1$ is proven in Theorem \ref{CC_coprime_case}. We now focus on the inductive step, depending on the type of Klein surface we are dealing with. As in Section \ref{notations}, we put $r' \coloneqq r/m$ and $d' \coloneqq d/m$, which implies that $d'/r'=d/r$ and $\gcd(r',d')=1$. 

\begin{case}
$\R X= \emptyset$.    
\end{case}

Let $(E_\R,h, \tau_\R)$ be a $C^\infty$ Hermitian vector bundle of rank $r$ and degree $d$ with real structure and let $\mathcal{A}_{YM}^{\, \tau_\R}/\mathcal{G}_h^{\tau_\R}$ be the moduli space of semistable real vector bundles that are smoothly isomorphic to $(E_\R,\tau_\R)$. By Proposition \ref{propcriteriaexist} and Theorem \ref{Lag_quot_connected}, this is a connected space, which is non-empty if and only $d$ is even. We simply denote by $f_\R : \mathcal{A}_{YM}^{\, \tau_\R}/\mathcal{G}_h^{\tau_\R} \to \R\mathcal{N}(r,d)$ the map $f_{\tau_\R}$ introduced in \eqref{map_from_Lag_quot}. In particular, $\Imm\, f_\R$ is connected. Similarly, let $(E_\H,h, \tau_\H)$ be a $C^\infty$ Hermitian vector bundle of rank $r$ and degree $d$ with quaternionic structure and let $\mathcal{A}_{YM}^{\, \tau_\H}/\mathcal{G}_h^{\tau_\H}$ be the moduli space of semistable quaternionic vector bundles that are smoothly isomorphic to $(E_\H,\tau_\H)$. By Proposition \ref{propcriteriaexist} and Theorem \ref{Lag_quot_connected}, this is a connected space, which is non-empty if and only $d + r(g-1)$ is even. We simply denote by $f_\H : \mathcal{A}_{YM}^{\, \tau_\H}/\mathcal{G}_h^{\tau_\H} \to \R\mathcal{N}(r,d)$ the map $f_{\tau_\H}$ introduced in \eqref{map_from_Lag_quot}.
\begin{remark}
    By construction, a semistable vector bundle $E\in\R\mathcal{N}(r,d)$ belongs to $\Imm\, f_\R$ if and only if admits a real structure, and it belongs to $\Imm\, f_\H$ if and only if admits a quaternionic structure.
\end{remark}

\textbf{Subcase 1: $g$ even.} In this case, $\R \Pic_d$ is connected for every $d$ (see Section \ref{picardwithoutrealpoints}). We thus need to show that $\R \mathcal{N}(r,d)$ is connected for all $r$ and $d$. Let us set $G_{\R}(r,d) \coloneqq  \Imm\, f_{\R},$ $G_{\H}(r,d) \coloneqq  \Imm\, f_{\H}$, as well as $G(r,d)\coloneqq G_{\R}(r,d) \cup G_{\H}(r,d)$ and start by showing that $G(r,d)$ is connected. By Proposition \ref{propcriteriaexist} and Theorem \ref{Lag_quot_connected}, we have that  $G_{\R}(r,d)$ is nonempty and connected if and only if $d$ is even and $G_{\H}(r,d)$ is nonempty and connected if and only if $r+d$ is even. In particular, $G_{\R}(r,d)$ and $G_{\H}(r,d)$ are both non-empty if and only if $r$ and $d$ are both even, so this is the only case in which there is something to prove. In this case, consider $\mathcal{F} \in \mathcal{N}^{st}(\frac{r}{2},
\frac{d}{2})$. By Remark \ref{remarkrealstability}, the bundle $\mathcal{E}\coloneqq \mathcal{F} \oplus \sigma^*(\overline{\mathcal{F}})$ is polystable and can be endowed with both a real and quaternionic structure. In particular, $\mathcal{E} \in  \Imm\, f_{\R} \cap \Imm\, f_{\H}$. This shows that, when both these sets are non-empty, they have non-empty intersection. Since $ \Imm\, f_{\R}$ and $ \Imm\, f_{\H}$ are connected, we deduce from the previous observation that $G(r,d)$ is connected for all $r$ and $d$. We will deduce from this that $\R \mathcal{N}(r,d)$ is also connected.

\smallskip

To prove it, observe first that, for every $(m_1,m_2) \in \N_{>0}^2$ such that $m_1+m_2=m$, we have a continuous map $$
\begin{array}{rcl}
f_{m_1,m_2}:\R \mathcal{N}(m_1r',m_1 d') \times \R \mathcal{N}(m_2 r',m_2d') & \longrightarrow & \R \mathcal{N}(r,d)\\
(\mathcal{E}_1,\mathcal{E}_2) & \longmapsto & \mathcal{E}_1 \oplus \mathcal{E}_2
\end{array}
$$
By Proposition \ref{polystable-real}, we have 
\begin{equation}\label{union1}
\R\mathcal{N}(r,d) \quad = \quad G(r,d)\quad \cup \bigcup_{\substack{(m_1,m_2) \in \N_{>0}^2 \\  m_1+m_2=m}} \Imm\, f_{m_1,m_2}.
\end{equation}
To show that $\R\mathcal{N}(r,d)$ is connected, it is thus enough to show that $$\bigcup_{\substack{(m_1,m_2) \in \N_{>0}^2 \\  m_1+m_2=m}} \Imm\, f_{m_1,m_2}$$ is connected and has non-empty intersection with $G(r,d)$. Before showing this, notice that, 
by Proposition \ref{propcriteriaexist} and Theorem \ref{Lag_quot_connected}, we have that $\R \mathcal{N}(r',d') \neq \emptyset$ if and only if $d'$ is even (in which case $r'$ is odd) or $r'+d'$ is even (in which case $r'$ and $d'$ are both odd). We thus have three situations to discuss.

\begin{enumerate}

\item If $d'$ is even (so $r'$ is odd, since $\gcd(r',d') = 1$), then the space $\R \mathcal{N}(r',d')$ is non-empty and contains only real bundles. Let $\mathcal{E} \in \R\mathcal{N}(r',d')$ be a (geometrically) stable one. The element $\mathcal{E}^{\oplus m}$ belongs to $\Imm\, f_{m_1,m_2}$ for every pair $(m_1,m_2) \in \N_{>0}^2 $such that $ m_1+m_2=m$. Since $m_i = \gcd(m_i r', m_i d') < m$, we have, by induction, that $\R\mathcal{N}(m_i r', m_i d')$ is connected for $i = 1, 2$. Therefore, the space $\Imm\, f_{m_1,m_2}$ is connected for all pairs $(m_1, m_2)$ such that $m_1+m_2=m$. Since every pairwise intersection of such spaces contains $\mathcal{E}^{\oplus m}$, this implies that  $$\bigcup_{\substack{(m_1,m_2) \in \N_{>0}^2 \\  m_1+m_2=m}} \Imm\, f_{m_1,m_2}$$ is connected. Moreover, since the vector bundle $\mathcal{E}^{\oplus m}$ can be endowed with a real structure, it also belongs to $G(r,d)$. From Equation (\ref{union1}), we deduce that $\R \mathcal{N}(r,d)$ is connected.

\item If $d'$ is odd and $r'$ is odd, then the space $\R \mathcal{N}(r',d')$ is non-empty and contains only quaternionic bundles. Let  $\mathcal{E} \in \R\mathcal{N}(r',d')$ be a (geometrically) stable one. Reasoning as above, we show that $\R \mathcal{N}(r,d)$ is connected in this case too.

\item If $d'$ is odd and $r'$ is even, then, as already noticed, we have $\R \mathcal{N}(r',d')=\emptyset$. This implies that, for every $p \in \N_{>0}$, the variety $\R \mathcal{N}(pr',pd')$ is non-empty if and only if $p$ is even. Indeed, consider $\mathcal{E} \in \R \mathcal{N}(pr',pd')$ and, using Proposition \ref{polystable-real}, write $$\mathcal{E}=\mathcal{E}_1 \oplus \cdots \oplus \mathcal{E}_k \oplus \big(\mathcal{F}_1 \oplus \sigma^*(\overline{\mathcal{F}}_1)\big) \oplus \cdots \oplus (\mathcal{F}_1 \oplus \sigma^*(\overline{\mathcal{F}}_l))$$ with $\mathcal{E}_i \in \R \mathcal{N}^{st}(s_ir',s_id')$ and $\mathcal{F}_i \in \mathcal{N}^{st}(h_ir',h_id')$ for some $s_1,\dots,s_k,h_1,\dots,h_l$. In particular, $\mathcal{F}_i \oplus \sigma^*(\overline{\mathcal{F}}_i) \in \mathcal{N}(2h_ir',2h_id')$ for every $i$ and $s_1+ \cdots +s_k+2h_1+\cdots +2h_l=p .$ By Proposition \ref{propcriteriaexist}, we have $\R \mathcal{N}^{st}(sr',sd')=\emptyset$ for every odd $s$, since $d'$ is odd and $r'$ is even. So every $s_i$ must be even and therefore also $p$. In particular, the number $m$ is also even. Write then $m=2m'$ and 
$$
\bigcup_{\substack{(m_1,m_2) \in \N_{>0}^2 \\  m_1+m_2=m}} \Imm\, f_{m_1,m_2} \quad = \bigcup_{\substack{(m'_1,m'_2) \in \N_{>0}^2 \\  m'_1+m'_2=m'}} \Imm\, f_{2m'_1,2m'_2}.
$$
Consider now $\mathcal{F} \in \mathcal{N}(r',d')$ and note that the element $(\mathcal{F} \oplus \sigma^*(\overline{\mathcal{F}}))^{\oplus m'}$ belongs to $\Imm\, f_{2m'_1,2m'_2}$ for every pair $(m'_1,m'_2)$ such that $m_1'+m_2'=m'$. This implies that $$\bigcup_{\substack{(m'_1,m'_2) \in \N_{>0}^2 \\  m'_1+m'_2=m}} \Imm\, f_{2m'_1,2m'_2}$$ is connected. Moreover, since the vector bundle $(\mathcal{F} \oplus \sigma^*(\overline{\mathcal{F}}))^{\oplus m'} $ can be endowed with a real structure, it also belongs to $G(r,d)$. From Equation \eqref{union1}, we then deduce that $\R \mathcal{N}(r,d)$ is connected.
\end{enumerate}

\textbf{Subcase 2: $g$ odd.} In this case, $\R \Pic_d$ is empty if $d$ is odd and $\pi_0(\R \Pic_d)=\{\R,\mathbb{H}\}$ if $d$ is even (see Section \ref{picardwithoutrealpoints}). In particular, $\R \mathcal{N}(r,d)=\emptyset$ if $d$ is odd. We can thus assume that $d$ is even and show that in this case $\R \mathcal{N}(r,d,\R)$ and $\R \mathcal{N}(r,d,\mathbb{H})$ are both connected. As in the first subcase, let us set $G_{\R}(r,d) \coloneqq  \Imm\, f_{\R}$ and $G_{\H}(r,d) \coloneqq  \Imm\, f_{\H}$. We then distinguish two cases. 
\begin{enumerate}
\item If $r$ is even, then by Remark \ref{remark-determinants}, we have 
$
G_{\R}(r,d) \cup G_{\H}(r,d) \subseteq \R \mathcal{N}(r,d,\R) := c_{r,d}^{-1}(\R).
$
In this case, we proceed as in the case when $g$ is even and define $G(r,d) := G_\R(r,d)\cup G_\H(r,d).$ Then, given $\mathcal{F} \in \mathcal{N}(\frac{r}{2},\frac{d}{2})$, the vector bundle $\mathcal{E} \coloneqq \mathcal{F} \oplus \sigma^*(\overline{\mathcal{F}})$ can be endowed with both a real and quaternionic structure, and from that we can deduce as in the first subcase that $G(r,d)$ is connected if $r$ is even. We will show below that this implies that $\R\mathcal{N}(r,d,\R)$ is connected when $r$ is even.
\item If $r$ is odd, then by Remark \ref{remark-determinants}, we have 
$
G_{\R}(r,d) \subseteq \R \mathcal{N}(r,d,\R) := c_{r,d}^{-1}(\R)$ and $G_{\mathbb{H}}(r,d) \subseteq \R \mathcal{N}(r,d,s_\mathbb{H}) := c_{r,d}^{-1}(\H)
$
with $G_{\R}(r,d)$ and $G_{\mathbb{H}}(r,d)$ both connected, and we will show below that this implies that $\R \mathcal{N}(r,d,\R)$ and $\R \mathcal{N}(r,d,\H)$ are connected when $r$ is odd, too.
\end{enumerate}
Let us first show that $\R \mathcal{N}(r,d,\R)$ is connected, for all $r$. For every pair $(m_1,m_2)$ such that $m_1+m_2=m$, we have a continuous map 
$$
\begin{array}{rcl}
f_{\R, m_1,m_2} : \R \mathcal{N}(m_1r',m_1d',\R) \times \R \mathcal{N}(m_2r',m_2d',\R) & \longmapsto & \R \mathcal{N}(r,d,\R) \\
(\mathcal{E}_1,\mathcal{E}_2) & \longmapsto & \mathcal{E}_1 \oplus \mathcal{E}_2 .
\end{array}
$$
\begin{lemma}\label{refine_decomp_g_odd}
If $r$ is even, 
\begin{equation*}
\R \mathcal{N}(r,d,\R) \quad = \quad G(r,d) \quad \cup \bigcup_{\substack{(m_1,m_2) \in \N_{>0}^2 \\ m_1+m_2=m}}\Imm\, f_{\R, m_1,m_2}   
\end{equation*}
and, if $r$ is odd, 
\begin{equation*}
\R \mathcal{N}(r,d,\R) \quad = \quad G_{\R}(r,d) \quad \cup \bigcup_{\substack{(m_1,m_2) \in \N_{>0}^2 \\ m_1+m_2=m}}\Imm\, f_{\R, m_1,m_2} \ .
\end{equation*}
\end{lemma}
\begin{proof}
Consider $\mathcal{E} \in \R \mathcal{N}(r,d,\R)$ such that, when $r$ is even, $\mathcal{E} \notin G(r,d)$ and, when $r$ is odd, $\mathcal{E} \notin G_{\R}(r,d) $. Using Propositions \ref{realquaternionicstability} and \ref{polystable-real}, write $$\mathcal{E}=\mathcal{E}_1 \oplus \cdots \oplus \mathcal{E}_k \oplus \big(\mathcal{F}_1 \oplus \sigma^*(\overline{\mathcal{F}}_1)\big) \oplus \cdots \oplus (\mathcal{F}_1 \oplus \sigma^*(\overline{\mathcal{F}}_l))$$ 
for some $\mathcal{E}_i \in \R \mathcal{N}^{st}(s_ir',s_id')$ and $\mathcal{F}_i \in \mathcal{N}^{st}(h_ir',h_id')$ and some integers $s_1,\dots,s_k,h_1,\dots,h_l$. If $\mathcal{E}$ does not belong to $G_{\R}(r,d)$ or $G(r,d)$, then there is a pair $(i, j)$ with $i\not= j$ such that $\mathcal{E}_i$ is real and $\mathcal{E}_j$ is quaternionic. Without loss of generality, we can assume that $\mathcal{E}_i \in \R \mathcal{N}^{st,\H}(s_ir',s_id')$ for $i=1,\dots,q$ and $\mathcal{E}_j \in \R \mathcal{N}^{st,\R}(s_j r',s_j d')$ for $j=q+1,\dots,k$, for some $q$ such that $1 \leq q <k$. Since $c_{r,d}(\mathcal{E})=\R$ and, for all $j = q + 1 , \dots, k$, $\det(\mathcal{E}_j)$ is a real line bundle, the line bundle $\det(\mathcal{E}_1 \oplus \cdots \oplus \mathcal{E}_q)$ must be real. Hence $\mathcal{E}_1 \oplus \cdots \oplus \mathcal{E}_q \in \R \mathcal{N}\big((s_1+\cdots +s_q)r',(s_1+\cdots +s_q)d',\R\big) $ and $\mathcal{E}_{q+1} \oplus \cdots \oplus \mathcal{E}_k\oplus  \big(\mathcal{F}_l \oplus \sigma^*(\overline{\mathcal{F}_l})\big) \in \R \mathcal{N}\big((m-s_q-\cdots -s_1)r',(m-s_q-\cdots -s_1)d', \R\big) .$ 
We thus have $\mathcal{E}=f_{\R, (s_1+\cdots +s_q),m-(s_1+\cdots +s_q)}\big(\mathcal{E}_1 \oplus \cdots \oplus \mathcal{E}_q,\mathcal{E}_{q+1} \oplus \cdots \oplus  \big(\mathcal{F}_l \oplus \sigma^*(\overline{\mathcal{F}_l})\big)\big)$ and $\mathcal{E} \in \Imm\, f_{\R, (s_1+\cdots +s_q),m-(s_1+\cdots +s_q)}$.
\end{proof}

\noindent To conclude the proof that $\R\mathcal{N}(r,d,\R)$ is connected, we distinguish two sub-cases. 
\begin{enumerate}
    \item Assume first that $d'$ is even. Then the variety $\R \mathcal{N}(r',d')$ is non-empty and contains a real bundle $\mathcal{E} \in \R\mathcal{N}(r',d',\R)$. The element $\mathcal{E}^{\oplus m}$ belongs belongs to $\Imm\, f_{\R,m_1,m_2}$ for every $(m_1,m_2) \in \N_{>0}^2$ such that $ m_1+m_2=m$. By the induction hypothesis, each $\Imm\, f_{m_1,m_2}$ is connected. This implies that $$\bigcup_{\substack{(m_1,m_2) \in \N_{>0}^2 \\  m_1+m_2=m}} \Imm\, f_{\R, m_1,m_2}$$ is connected. Moreover, since the vector bundle $\mathcal{E}^{\oplus m}$ can be endowed with a real structure, it belongs to $G_{\R}(r,d)$. We deduce that $\R \mathcal{N}(r,d,\R)$ is connected in this case.
    \item Assume next that $d'$ is odd. Since $d= md'$ and $d$ is even, we have that $m$ is even, too. Write then $m=2m'$. Using a similar argument to the one used when $\R X = \emptyset$ and $g$ even, we can argue that $\R \mathcal{N}(sr',sd')$ is empty when $s$ is odd, hence that
    $$
    \bigcup_{\substack{(m_1,m_2) \in \N_{>0}^2 \\  m_1+m_2=m}} \Imm\, f_{\R, m_1,m_2}\qquad = \quad \bigcup_{\substack{(m'_1,m'_2) \in \N_{>0}^2 \\  m'_1+m'_2=m}} \Imm\, f_{\R, 2m'_1,2m'_2}\ .
    $$
    Consider now $\mathcal{F} \in \mathcal{N}(r',d')$. Notice that the element 
    $(\mathcal{F} \oplus \sigma^*(\overline{\mathcal{F}}))^{\oplus m'} $ belongs to 
    $\Imm\, f_{\R, 2m'_1,2m'_2}$ for every $m'_1,m'_2$ such that $m_1'+m_2'=m'$. This 
    implies that 
    $$
    \bigcup_{\substack{(m'_1,m'_2) \in \N_{>0}^2 \\  m'_1+m'_2=m}} \Imm\, f_{\R, 
    2m'_1,2m'_2}
    $$ is connected. Moreover, since the vector bundle $(\mathcal{F} \oplus \sigma^*(\overline{\mathcal{F}}))^{\oplus m'} $ can be endowed with a real structure, it belongs to $G_{\R}(r,d)$ too. We deduce that $\R \mathcal{N}(r,d,\R)$ is connected in this case too.
\end{enumerate}

Let us now show that $\R \mathcal{N}(r,d,\mathbb{H})$ is connected.  For every pair $(m_1,m_2)$ such that $m_1+m_2=m$, we have a morphism 
$$
\begin{array}{rcl}
f_{\H,m_1,m_2} : \R \mathcal{N}(m_1r',m_1d',s_\mathbb{H}) \times \R \mathcal{N}(m_2r',m_2d',\R) & \longrightarrow & \R \mathcal{N}(r,d,s_\mathbb{H}) \\
(\mathcal{E}_1,\mathcal{E}_2) & \longmapsto & \mathcal{E}_1 \oplus \mathcal{E}_2 
\end{array}.
$$ As in Lemma \ref{refine_decomp_g_odd}, we have, if $r$ is even,
\begin{equation*}
\R \mathcal{N}(r,d,\mathbb{H}) \quad = \quad \bigcup_{\substack{(m_1,m_2) \in \N_{>0}^2 \\ m_1+m_2=m}}\Imm\, f_{\H, m_1,m_2} 
\end{equation*}
and, if $r$ is odd, 
\begin{equation*}
\R \mathcal{N}(r,d,\mathbb{H}) \quad = \quad G_{\mathbb{H}}(r,d) \quad \cup \bigcup_{\substack{(m_1,m_2) \in \N_{>0}^2 \\ m_1+m_2=m}}\Imm\, f_{\H, m_1,m_2}\ .
\end{equation*}
We treat each case separately. 
\begin{enumerate}
    \item Assume that $r$ is odd. Then, as $d$ is even, $m := \gcd(r,d)$ must be odd and, as $d = md'$, $d'$ must be odd. Let $m=2h+1$. In particular, for every pair $(m_1,m_2)$ such that $m_1+m_2=m$, either $m_1$ is odd and $m_2$ or even or vice-versa. Assume for instance that $m_1$ is odd with $m_1=2h_1+1$ and $m_2=2h_2$. Fix a real vector bundle $\mathcal{G} \in \R \mathcal{N}(r',d',\R)$ and a quaternionic vector bundle $\mathcal{E} \in \R \mathcal{N}(r',d',s_\mathbb{H})$. For the polystable vector bundle $\mathcal{W} \coloneqq \mathcal{G}^{\oplus 2h} \oplus \mathcal{E}$, we have $\mathcal{W}=f_{\H, m_1,m_2}\big(\mathcal{G}^{2h_1} \oplus \mathcal{E},\mathcal{G}^{2h_2}\big)$ and $\mathcal{W}=f_{\H, m_2,m_1}\big(\mathcal{G}^{2h_2-1} \oplus \mathcal{E},\mathcal{G}^{2h_1+1}\big)\ .$ In particular, $\Imm\, f_{\H, m_1,m_2} \cap \Imm\, f_{\H,m_2,m_1} \neq \emptyset$ and thus, by induction, for all $m_1,m_2$ such that $m_1+m_2=m$, we have that $\Imm\, f_{\H, m_1,m_2} \cup \Imm\, f_{\H,m_2,m_1}$ is connected. Notice that $$\bigcup_{\substack{(m_1,m_2) \in \N_{>0}^2 \\ m_1+m_2=m}}\Imm\, f_{\H, m_1,m_2} \quad = \bigcup_{\substack{(m_1,m_2) \in \N_{>0}^2 \\ m_1+m_2=m\\m_1 \text{ is odd}}} \Imm\, f_{\H, m_1,m_2}  \cup \Imm\, f_{\H,m_2,m_1} .$$
    For every pair $(m_1,m_2)$ such that $m_1+m_2=m$ and $m_1$ is odd, we have that $\mathcal{E}^{\oplus m} = f_{\H, m_1,m_2}(\mathcal{E}^{\oplus m_1},\mathcal{E}^{\oplus m_2})$. Since each $\Imm\, f_{\H, m_1,m_2}  \cup \Imm\, f_{\H,m_2,m_1}$ with $m_1$ odd is connected, we deduce that $$\bigcup_{\substack{(m_1,m_2) \in \N_{>0}^2 \\ m_1+m_2=m}}\Imm\, f_{\H, m_1,m_2}$$ is also connected. The bundle $\mathcal{E}^{\oplus m}$ can be endowed with a quaternionic structure, thus it belongs to $G_{\mathbb{H}}(r,d)$. We deduce from this that $\R \mathcal{N}(r,d,\mathbb{H})$ is connected.
    \item Assume now that $r$ is even. We distinguish two sub-cases. 
    \begin{itemize}
        \item Assume first that $d'$ is odd. Then, as shown above, $\R \mathcal{N}(sr',sd')=\emptyset$ if $s$ is odd. In particular, for all $\mathcal{E} \in \R \mathcal{N}(r,d)$, by Propositions \ref{realquaternionicstability} and \ref{polystable-real}, we can write $$\mathcal{E}=\mathcal{E}_1 \oplus \cdots \oplus \mathcal{E}_k \oplus (\mathcal{F}_1 \oplus \sigma^*(\overline{\mathcal{F}}_1)) \oplus \cdots \oplus (\mathcal{F}_1 \oplus \sigma^*(\overline{\mathcal{F}}_l))$$ for some geometrically stable bundles $\mathcal{E}_i \in \R \mathcal{N}^{st}(s_ir',s_id')$ and $\mathcal{F}_i \in \mathcal{N}^{st}(h_ir',h_id')$ and some integers $s_1,\dots,s_k,h_1,\dots,h_l$. We must therefore have that $s_1,\dots,s_k$ are even. This implies that $\mathcal{E}_i \in \R \mathcal{N}(s_ir',s_id',\R)$ for each $i$ and thus $\mathcal{E} \in \R \mathcal{N}(r,d,\R)$ too, i.e.\ $\R \mathcal{N}(r,d,s_\mathbb{H})= \emptyset$ in this case (in particular, it is connected). Note that the property $c_{r,d}^{-1}(\H) = \emptyset$ does not contradict the fact, already seen, that there quaternionic line bundles of rank $r$ and degree $d$ in the case at hand, including geometrically stable ones. It just so happens that, as $r$ is even, the determinant of such a bundle is a real line bundle, so quaternionic vector bundles are contained in the fibre $c_{r,d}^{-1}(\R)$ in this case.
        \item Assume now that $d'$ is even. If $m$ is odd, we can argue as above to deduce the connectedness of $$ \bigcup_{\substack{(m_1,m_2) \in \N_{>0}^2 \\ m_1+m_2=m}}\Imm\, f_{\H, m_1,m_2} .$$ So we assume from now on that $m$ is even, with $m=2h$. By Propositions \ref{realquaternionicstability} and \ref{polystable-real}, we can write $\mathcal{E}=\mathcal{E}_1 \oplus \cdots \oplus \mathcal{E}_k \oplus \big(\mathcal{F}_1 \oplus \sigma^*(\overline{\mathcal{F}}_1)\big) \oplus \cdots \oplus (\mathcal{F}_1 \oplus \sigma^*(\overline{\mathcal{F}}_l))$ for some geometrically stable bundles $\mathcal{E}_i \in \R \mathcal{N}^{st}(s_ir',s_id')$ and $\mathcal{F}_i \in \mathcal{N}^{st}(h_ir',h_id')$ and some integers $s_1,\dots,s_k,h_1,\dots,h_l$. We can moreover assume without loss of generality that $\mathcal{E}_i\in \R \mathcal{N}^{st}(s_ir',s_id',\mathbb{H})$ if $i=1,\dots,q$ and $\mathcal{E}_i \in \R \mathcal{N}(s_ir',s_id',\R)$ if $i=q+1,\dots,k$. As remarked above, this implies that $s_i$ is odd for each $i=1,\dots,q$ and that $q$ is odd. We deduce from it that $$\R \mathcal{N}(r,d,\mathbb{H}) \quad= \bigcup_{\substack{(m_1,m_2) \in \N_{>0}^2 \\ m_1+m_2=m}}\Imm\, f_{\H, m_1,m_2} \quad =\bigcup_{\substack{(m_1,m_2) \in \N_{>0}^2 \\ m_1+m_2=m \\ m_1,m_2\ \text{both odd}}}\Imm\, f_{\H, m_1,m_2}\ .$$ For every pair $(m_1,m_2)$ such that $m_1+m_2=m$ and $m_1,m_2$ both odd, let $m_1=2h_1+1$ and $m_2=2h_2-1$, i.e.\ $h_1+h_2=h$. Fix $\mathcal{G} \in \R \mathcal{N}(r',d',\R)$ and $\mathcal{E} \in \R \mathcal{N}(r',d',\mathbb{H})$. Let $\mathcal{W} \coloneqq \mathcal{E} \oplus \mathcal{G}^{\oplus 2h-1} .$ For each $m_1,m_2$ as above, we have $\mathcal{W}=f_{\H, m_1,m_2}(\mathcal{E} \oplus \mathcal{G}^{\oplus 2h_1},\mathcal{G}^{2h_2-1}) .$ We deduce that 
        $$\bigcap_{\substack{(m_1,m_2) \in \N_{>0}^2 \\ m_1+m_2=m \\ m_1,m_2\ \text{both odd}}}\Imm\, f_{\H, m_1,m_2} \neq \emptyset$$ so, by induction, $\R \mathcal{N}(r,d,s_\mathbb{H})$ is connected.
    \end{itemize}
\end{enumerate}

\begin{case}
$\R X \neq \emptyset$ (i.e.\ $n >0$).   
\end{case}

Let $r \geq 1$ and let $d\in \Z$. By section \ref{picardwithrealpoints}, the $2^{n-1}$ connected components of the variety $\R \Pic_d$ are indexed by the elements $s=(s_1,\dots,s_n) \in (\Z/2\Z)^{n}$ such that $s_1+\cdots +s_n \equiv d \ \mathrm{mod}\ 2$. Given an $s \in (\Z/2\Z)^n$ satisfying that condition, Proposition \ref{propcriteriaexist} shows that there exists a $C^\infty$ real vector bundle $(E,\tau_s)$, unique up to isomorphism of such bundles, such that $w_1(E^{\tau_s}) = s$. We denote by $\mathcal{A}_{YM}^{\, \tau_s}/\mathcal{G}_h^{\tau_s}$ the moduli space of semistable real vector bundles that are smoothly isomorphic to $(E_\R,\tau_s)$. By Proposition \ref{propcriteriaexist} and Theorem \ref{Lag_quot_connected}, this is a non-empty and connected space. We simply denote by $f_{\R,s} : \mathcal{A}_{YM}^{\, \tau_s}/\mathcal{G}_h^{\tau_s} \to \R\mathcal{N}(r,d)$ the map $f_{\tau_s}$ introduced in \eqref{map_from_Lag_quot}. We then denote by $G_{\R,s}(r,d)$ the connected subspace
$
G_{\R,s}(r,d) \coloneqq \Imm\, f_{\R,s} \subseteq \R \mathcal{N}(r,d,s)\, .
$ 
When $\R X \neq \emptyset$, Proposition \ref{propcriteriaexist} shows that there exists a $C^\infty$ Hermitian vector bundle $(E, \tau_H)$ of rank $r$ and degree $d$ if and only if $r$ and $d$ are both even, and that such a bundle is unique up to isomorphism. In that case, Theorem \ref{Lag_quot_connected} shows that the moduli space $\mathcal{A}_{YM}^{\, \tau_s}/\mathcal{G}_h^{\tau_s}$ of semistable quaternionic vector bundles that are smoothly isomorphic to $(E,\tau_\H)$ is non-empty and connected and we shall simply denote by $f_{\H} : \mathcal{A}_{YM}^{\, \tau_\H}/\mathcal{G}_h^{\tau_\H} \to \R\mathcal{N}(r,d)$ the map $f_{\tau_\H}$ introduced in \eqref{map_from_Lag_quot}. The next observation is extracted from the proof of \cite[Theorem~2.5]{Sch-Sco}.
\begin{lemma}\label{SW_class_for_quat_bundles}
Let $(E,\tau_\H)$ be a quaternionic line bundle of rank $2r$ and degree $2d$ on a real curve $X$ such that $\R X \not=\emptyset$. Then the determinant bundle $\det E$, which is a real line bundle, satisfies $w_1(\R E) = (0, \dots , 0) \in (\Z/2\Z)^n$.
\end{lemma}
\begin{proof}
    Observe that $\det (\R E) = \R \det E = \R \det(E|_{\R X})$ and that $E_{\R X}$ is a complex vector bundle with quaternionic structure over a real space with trivial real structure in the sense of Atiyah \cite{Atiyah_real}. So $E|_{\R X}$ admits a reduction of structure group to $\mathbf{Sp}(2r,\C) \subset \mathbf{SL}(2r,\C)$ and this reduction is compatible with its real structure. The line bundle $\det(E|_{\R X})$ is therefore trivial as a complex line bundle with real structure, which implies that $w_1(\R E) = (0, \dots, 0)\in(\Z/2\Z)^n$.
\end{proof}
By Lemma \ref{SW_class_for_quat_bundles}, if we denote by $G_{\H}(2r,2d)$ the connected subspace $\Imm\, f_\H \subset \R\mathcal{N}(2r,2d)$, we have an inclusion
\begin{equation}\label{connected_chunks}
G_{\mathbb{H}}(2r,2d) \coloneqq  \Imm\, f_\H \subseteq \R \mathcal{N}(2r,2d,0) := c_{2r,2d}^{-1}(0)\, .
\end{equation}
Fix now $s \in (\Z/2\Z)^n$ such that $s_1+\cdots +s_n \equiv d \ \mathrm{mod}\ 2$ and put $G_s(r,d) \coloneqq G_{\R,s}(r,d) \cup G_{\mathbb{H},s}(r,d) ,$ where $G_{\H,0}(r,d) := G_\H(r,d)$ if $r$ and $d$ are both even, and $G_{\mathbb{H},s}(r,d)\coloneqq \emptyset$ if $s \neq 0$ or $r$ and $d$ not both even. To show that $\R\mathcal{N}(r,d,s)$ is connected, we first show that $G_s(r,d)$ is connected. Since the spaces $\Imm\, f_{\R,s}$ and $\Imm\, f_{\H}$ are all connected, the only case in which there is something to prove is when $r$ and $d$ are both even, and in that case it suffices to show that $G_{\R,0}(r,d) \cap G_{\mathbb{H}}(r,d) \neq \emptyset$. Consider then $\mathcal{F} \in \mathcal{N}^{st}(\frac{r}{2},\frac{d}{2})$. Since the bundle $\mathcal{E} \coloneqq \mathcal{F} \oplus \sigma^*(\overline{\mathcal{F}}) $ is polystable and can be endowed with both a real and quaternionic structure, it belongs to $G_{\R,0}(r,d) \cap G_{\mathbb{H}}(r,d)$ and this concludes the proof that $G_0(r,d) = G_{\R,0}(r,d) \cap G_{\mathbb{H}}(r,d)$ is connected. We will deduce from this that $\R\mathcal{N}(r,d,s)$ is also connected. But first observe that, for all $r$ and $d$ and for every pair $(m_1,m_2) \in \N_{>0}$ such that $m_1+m_2=m$ and $m_2 d'\equiv d \ \mathrm{mod}\ 2$, there is a continuous map 
$$
\begin{array}{rcl}
f_{s, m_1,m_2} : \R\mathcal{N}(m_1r',m_1d',0) \times\R\mathcal{N}(m_2r',m_2d',s) & \longrightarrow & \R \mathcal{N}(r,d,s)\\
(\mathcal{E}_1,\mathcal{E}_2) & \longmapsto & \mathcal{E}_1 \oplus\mathcal{E}_2 .
\end{array}
$$
Consider now an element $\mathcal{E} \in \R \mathcal{N}(r,d,s)$. By Proposition \ref{polystable-real}, we can write $$\mathcal{E}=\mathcal{E}_1 \oplus \cdots \oplus \mathcal{E}_k \oplus \big(\mathcal{F}_1 \oplus \sigma^*(\overline{\mathcal{F}}_1)\big) \oplus \cdots \oplus (\mathcal{F}_1 \oplus \sigma^*(\overline{\mathcal{F}}_l))$$ with $\mathcal{E}_i \in \R \mathcal{N}^{st}(e_ir',e_id')$ and $\mathcal{F}_i \in \mathcal{N}^{st}(h_ir',h_id')$ for some $e_1,\dots,e_k,h_1,\dots,h_l$. Assume moreover that $\mathcal{E}_i \in \R \mathcal{N}^{st,\mathbb{H}}(e_ir',e_id')$ for $i=1,\dots,q$ and $\mathcal{E}_i \in \R \mathcal{N}^{st,\mathbb{R}}(e_ir',e_id')$ for $i=q+1,\dots,k$. Put $m':=e_1+ \cdots + e_q + 2h_1 +\cdots+2h_l$. Notice that we have $$\mathcal{E}_1 \oplus \cdots \oplus \mathcal{E}_q \oplus \big(\mathcal{F}_1 \oplus \sigma^*(\overline{\mathcal{F}}_1)\big) \oplus \cdots \oplus \big(\mathcal{F}_1 \oplus \sigma^*(\overline{\mathcal{F}}_l)\big) \in \R \mathcal{N}(m'r',m'd',0)$$ and thus $\mathcal{E}_{q+1} \oplus \cdots \oplus \mathcal{E}_k \in \R \mathcal{N}((m-m')r',(m-m')d',s) .$ Notice moreover that $\mathcal{E}_{q+1} \oplus \cdots \oplus \mathcal{E}_k$ can be endowed with a real structure, i.e.\ $\mathcal{E}_{q+1} \oplus \cdots \oplus \mathcal{E}_k \in G_{s}(r,d)$. In particular, if $m'=0$, we have $\mathcal{E} \in G_{s}(r,d)$. We deduce from it that, for every $s$ as above, we have 
\begin{equation*}
    \R \mathcal{N}(r,d,s) \quad = \quad G_s(r,d)\quad \cup \bigcup_{\substack{(m_1,m_2) \in \N_{>0}^2 \\ m_1+m_2=m \\ m_2d' \equiv d \ \mathrm{mod}\ 2}}\Imm\, f_{s,m_1,m_2}
\end{equation*}
We then have two cases to analyze: $d'$ even and $d'$ odd.
\begin{enumerate}

\item If $d'$ is even, then $d$ is even too and $m_2 d' \equiv d \equiv 0 \ \mathrm{mod}\ 2$ for every $m_1,m_2$ such that $m_1+m_2=m$. In this case, consider a geometrically stable real bundle $\mathcal{G} \in \R \mathcal{N}^{st,\R}(r',d',0)$ and a (geometrically) stable real bundle $\mathcal{E} \in \R \mathcal{N}^{st}(r',d',s)$. Put $\mathcal{W} \coloneqq \mathcal{G}^{m-1} \oplus \mathcal{E} .$ For every $(m_1,m_2)$ such that $m_1+m_2=m$, we have 
$
\mathcal{W}=f_{s,m_1,m_2}(\mathcal{G}^{\oplus m_1},\mathcal{G}^{m_2-1} \oplus \mathcal{E}).
$
By induction, each $\Imm\, f_{s,m_1,m_2}$ is connected, from which we deduce that 
$$\bigcup_{\substack{(m_1,m_2) \in \N_{>0}^2 \\ m_1+m_2=m \\ m_2d' \equiv d \ \mathrm{mod}\ 2}}\Imm\, f_{s,m_1,m_2}$$ is connected, too. Notice that $\mathcal{W}$ can be endowed with a real structure, so it belongs to $G_s(r,d)$. Thus $\R \mathcal{N}(r,d,s)$ is also connected.

\item If $d'$ is odd, we have two further subcases. 
    \begin{enumerate}
    \item If $d$ is also odd, i.e.\ if $m$ is odd, then for all $m_1,m_2$ such that $m_1+m_2=m$ and $m_2 d' \equiv d \ \mathrm{mod}\ 2$, we must have $m_1$ even and $m_2$ odd. Put $m=2m'+1$, $m_1=2m_1'$ and $m_2=2m_2'+1$. Consider a geometrically stable real bundle $\mathcal{G} \in \R \mathcal{N}^{st,\R}(2r',2d',0)$ and a (geometrically) stable real bundle $\mathcal{E} \in \R \mathcal{N}(r',d',s)$. Let
    $
    \mathcal{W} \coloneqq \mathcal{G}^{m'} \oplus \mathcal{E} .
    $
    Then $\mathcal{W}=f_{s,m_1,m_2}(\mathcal{G}^{\oplus m_1'},\mathcal{G}^{\oplus m_2'} \oplus \mathcal{E})$ for all $m_1,m_2$ such that $m_2$ is odd and $m_1+m_2=m$. By induction, each $\Imm\, f_{s,m_1,m_2}$ is connected and we deduce from it that 
    $$
    \bigcup_{\substack{(m_1,m_2) \in \N_{>0}^2 \\ m_1+m_2=m \\ m_2d' \equiv d \ \mathrm{mod}\ 2}}\Imm\, f_{s,m_1,m_2}
    $$
    is connected, too. Notice that $\mathcal{W}$ can also be endowed with a real structure, so it belongs to $G_s(r,d)$ as well. It follows that $\R \mathcal{N}(r,d,s)$ is connected.
    \item If $d$ is even then $m$ is even. Put $m=2m'$. Notice that given $m_1,m_2$ such that $m_2d' \equiv d \ \mathrm{mod}\ 2$ and $m_1+m_2=m$, we must have that $m_1,m_2$ are both even. Put $m_1=2m_1'$ and $m_2=2m_2'$. Consider geometrically stable real bundles $\mathcal{G} \in \R \mathcal{N}^{st}(2r',2d',0)$ and $\mathcal{E} \in \R\mathcal{N}^{st}(2r',2d',s)$. Define $\mathcal{W} \coloneqq \mathcal{G}^{\oplus m'-1} \oplus \mathcal{E}$. For all $m_1,m_2$, we then have $\mathcal{W}=f_{s,m_1,m_2}(\mathcal{G}^{\oplus m_1'},\mathcal{G}^{\oplus m_2'-1} \oplus \mathcal{E})$. Then the same argument as above shows that $\R \mathcal{N}(r,d,s)$ is connected.
    \end{enumerate}
\end{enumerate}
\end{proof}

When $\R X = \emptyset$, we summarize the results of the current section in Table \ref{CC_count_curves_without_real_pts}, which gives the number of connected components of $\R\mathcal{N}(r,d)$ for all $r$ and $d$. When $r$ and $d$ are coprime, the results are consistent with \cite[Theorem~1.1]{Sch_JSG}. As seen in Theorem \ref{CC_coprime_case}, when $\gcd(r,d) = 1$, the determinant morphism induces a bijection bijection $\pi_0(\R\mathcal{N}(r,d)) \simeq \pi_0(\R\Pic_d)$. But for general $r$ and $d$, it can happen that the map $c_{r,d} : \R\mathcal{N}(r,d) \to \pi_0(\R\Pic_d)$ is not surjective when $\R X = \emptyset$, and Table \ref{CC_count_curves_without_real_pts} below shows that this happens precisely when:
\begin{itemize}
    \item $\R X = \emptyset$, $g$ even, $d$ odd and $r$ even.
    \item $\R X = \emptyset$, $g$ odd, $d$ even and $r$ even.
\end{itemize}
Note that $c_{2r,2d}^{-1}(\R)$ contains quaternionic vector bundles, both for $g$ even and $g$ odd.
\begin{table}[!h]
    \centering{
    \begin{tabular}{|c|c|c|c|c|c|c|}
        \hline
        $g$ & $d$ & $r$ & $\R\mathcal{N}(r,d)$ & $|\pi_0(\R\mathcal{N}(r,d))|$ & $\R\Pic_d$ & $|\pi_0(\R\Pic_d)|$  \\
        \hline
          even & even & even & $c_{2r,2d}^{-1}(\R)$ & 1 & $\Pic_{2d}^\R$ & 1 \\
        \hline
          even & even & odd & $c_{2r+1,2d}^{-1}(\R)$ & 1 & $\Pic_{2d}^\R$ & 1 \\
        \hline
          even & odd & even & $\emptyset$ & 0 & $\Pic_{2d+1}^\H$ & 1 \\
        \hline
          even & odd & odd & $c_{2r+1,2d+1}^{-1}(\H)$ & 1 & $\Pic_{2d+1}^\H$ & 1 \\
        \hline
          odd & even & even & $c_{2r,2d}^{-1}(\R)$ & 1 & $\Pic_{2d}^\R \sqcup \Pic_{2d}^\H$ & 2 \\
        \hline
          odd & even & odd & $c_{2r+1,2d}^{-1}(\R) \sqcup c_{2r+1,2d}^{-1}(\H)$ & 2 & $\Pic_{2d}^\R \sqcup \Pic_{2d}^\H$ & 2 \\
        \hline
          odd & odd & even & $\emptyset$ & 0 & $\emptyset$ & 0 \\
        \hline
          odd & odd & odd & $\emptyset$ & 0 & $\emptyset$ & 0 \\
        \hline
    \end{tabular}
    \medskip
    \caption{Connected components of $\R \mathcal{N}(r,d)$ when $\R X = \emptyset$.}
    \label{CC_count_curves_without_real_pts}
    }
\end{table}

When $\R X \not=\emptyset$, the real loci $\R\mathcal{N}(r,d)$ and $\R\Pic_d$ both have $2^{n-1}$ connected components, where $n$ is the number of connected components of $\R X$, but $\R\mathcal{N}(2r,2d)$ contains quaternionic bundles, too, not just real bundles. By Lemma \ref{SW_class_for_quat_bundles}, these quaternionic vector bundles are all contained in the fibre $\R\mathcal{N}(2r,2d,0) := c_{2r,2d}^{-1}(0)$.

\begin{teorema}\label{CC_count_curves_with_real_pts}
    Let $X$ be a Klein surface of genus $g\geqslant 2$ such that $\R X$ has $n > 0$ connected components. Let $\mathcal{N}(r,d)$ be the moduli space of semistable vector bundles of rank $r\geq 1$ and degree $d\in\Z$ on $X$, endowed with its canonical real structure. Then the determinant morphism $\det : \R\mathcal{N}(r,d) \to \R\Pic_d$ induces a bijection $\pi_0(\R\mathcal{N}(r,d)) \simeq \pi_0(\R\Pic_d)$. In particular, the real locus $ \R\mathcal{N}(r,d) $ has $2^{n-1}$ connected components.
\end{teorema}

\begin{remark}
While the statement of Theorem \ref{mainresult} can be formulated for any curve $X$, its proof as stated above holds only under the condition that the genus $g$ of the curve $X$ is greater than or equal to $2$. 
In the case when $g=0$, i.e.\ $X=\mathbb{P}^1_{\C}$, the set $\mathcal{N}(r,d)$ is non-empty if and only if $r \mid d$, in which case it consists of a single point. There are two possible real structures on $\mathbb{P}^1_{\C}$, namely $[z_0 : z_1]\mapsto [\ov{z_0}, \ov{z_1}]$ and $[z_0 : z_1] \mapsto [-\ov{z_1} : \ov{z_0}]$, and in both cases, $\R\mathcal{N}(r,d)$ is connected and non-empty whenever $r \mid d$. In particular, $\R \Pic_d$ is always a point and we see that Theorem \ref{mainresult} still holds when $g=0$. The case when $g=1$ is  more involved. For a description of the real points of $\R\mathcal{N}(r,d)$ in the case when $X$ is a real elliptic curve, see for instance \cite{BS}. And for the connectedness, we proved in \cite{Sch-Sco} that, if $\gcd(r,d)=2$, then for every $s \in \pi_0(\mathbb{R}\Pic_d)$, $\R\mathcal{N}(r,d,s)$ is connected. In particular, Theorem \ref{mainresult} extends to this case as well. 
\end{remark}

\section{Branes in Hitchin moduli spaces} \label{sectionà-Higgs}

\subsection{Moduli spaces of Higgs bundles}

A Higgs bundle on $X$ is a pair $(E,\phi)$ consisting of a holomorphic vector bundle $E$ on $X$ and a morphism of $O_X$-modules $\phi : E \to E \otimes \Omega^1_X$, known as a Higgs field. The degree  and rank of a Higgs bundle $(E,\phi)$ are the degree and rank of the underlying vector bundle $E$. A morphism of Higgs bundles $f:(E_1,\phi_1) \to (E_2,\phi_2)$ is a morphism of vector bundles $f:E_1 \to E_2$ that makes the following diagram commute.
$$
\begin{tikzcd}
    E_1 \ar[d,"f"'] \ar[r, "\phi_1"'] & E_1 \otimes \Omega^1_X\ar[d, "f \otimes \mathrm{Id}_{\Omega^1_X}"']  \\
    E_2 \ar[r, "\phi_2"'] & E_2 \otimes \Omega^1_X
\end{tikzcd}
$$
The Higgs bundle $(E,\phi)$ is called semistable if for all non-zero sub-bundle $F \subset $ such that $\phi(F) \subset F \otimes \Omega^1_X$, $\mu(F) \leq \mu(E)$. For all $r$ and $d$, there exists a quasi-projective complex algebraic variety $\MD(r,d)$,  of dimension $2(r^2(g-1) + 1)$, parameterizing $S$-equivalence classes of semistable Higgs bundles of rank $r$ and degree $d$ over $X$. Following Simpson \cite{Simpson_local_systems}, we will refer to these varieties as Dolbeault moduli spaces. 

\smallskip

The real structure $\sigma : X \to X$ induces a real structure on $\MD(r,d)$, defined by
\begin{equation}\label{Dol_can_real_structure}
\sigma([E,\phi])=[\sigma^*(\overline{E}),\sigma^*(\overline{\Phi})].
\end{equation}
This is an anti-holomorphic involution whose set of real points will be denoted by $\R \MD(r,d)$. If $\gcd(r,d) = 1$, it has a canonical structure of real analytic manifold of real dimension $2(r^2(g-1) + 1)$. The map $c_{r,d}$ defined in \eqref{obstruction_map} extends to a map
\begin{equation}\label{obstruction_map_Dolbeault}
\begin{array}{rcl}
    \widehat{c_{r,d}} : \R\MD(r,d) & \longrightarrow & \pi_0(\R\Pic_d) \\
    (E,\phi) & \longmapsto & [\det(E)]
\end{array}
\end{equation}
and we set, for all $s \in \pi_0(\R \Pic_d)$, $\R \MD(r,d,s) := \widehat{c_{r,d}}^{-1}(s)$. We then have an open-closed decomposition
$$
\R \MD(r,d) = \bigsqcup_{c \in \Imm{\widehat{c_{r,d}}}} \widehat{c_{r,d}}^{-1}(s)
$$
and we will show that this is in fact the decomposition of the real locus $\R\MD(r,d)$ into connected components. More precisely, we will show that $\Imm{\widehat{c_{r,d}}} = \Imm{c_{r,d}}$ and that, for all $s \in \pi_0(\R\Pic_d)$, the space $\R\MD(r,d,s)$ is connected. Since $\R\mathcal{N}(r,d)$ can be identified with the subset of all $(E,\phi) \in \R\MD(r,d)$ for which $\phi = 0$, it suffices to show Theorem \ref{theoremHiggsbundles} below. This improves on \cite{Sch-Sco}, where the result was shown under the assumption that $r$ and $d$ were coprime, in which case there is a modular description of the varieties $\R \MD(r,d,s)$ as moduli spaces of real/quaternionic Higgs bundles of fixed topological type.

\begin{teorema}\label{theoremHiggsbundles}
For all $s \in \pi_0(\R \Pic_d)$, the subspace $\R \MD(r,d,s)$ is the connected component of $\R\MD(r,d)$ containing $\R\mathcal{N}(r,d,s)$.
\end{teorema}

The result above is a consequence of \cite[Theorem 3.11]{Sch-Sco} and Theorem \ref{mainresult}. We give a sketch of proof here for the sake of completeness.

\begin{proof}
The idea is to show the existence of a path between a Galois-invariant Higgs bundle $(E,\phi) \in \R\MD(r,d,s)$ and the (connected) subvariety $\R\mathcal{N}(r,d,s)$. The difference with \cite{Sch-Sco} is that, in that paper, we only knew that $\mathcal{N}(r,d,s)$ was connected when $\gcd(r,d) = 1$ (but the rest of the argument of \cite[Theorem 3.11]{Sch-Sco} did not use that assumption). The path between the Galois-invariant Higgs bundle $(E,\phi) \in \R\MD(r,d,s)$ and $\R\mathcal{N}(r,d,s)$ is constructed via a series of real Bialynicki-Birula flows $f : \R\mathbf{P}^1 \to \R\MD(r,d,s)$, the point being that these stay contained in the fibre $\widehat{c_{r,d}}^{-1}(s)$ from which they start (\cite[Proposition~3.10]{Sch-Sco}). To construct these real flows, we need to understand the structure of $\R^*$-fixed points in $\R\MD(r,d,s)$ (see \cite[Theorem~3.8]{Sch-Sco}). We then obtain the inclusion $\Imm\,\widehat{c_{r,d}} \subset \in \Imm\, c_{r,d}$, the converse inclusion being obvious since $\widehat{c_{r,d}}|_{\R\mathcal{N}(r,d)} = c_{r,d}$. Since $\R\mathcal{N}(r,d,s)$ is connected, any two points of $\R\MD(r,d,s)$ can be related via a path, as illustrated in Figure \ref{fig_realBBflows}. So $\R\MD(r,d,s)$ is connected and it is the connected component of $\R\MD(r,d)$ containing $\R\mathcal{N}(r,d,s)$.
\end{proof}

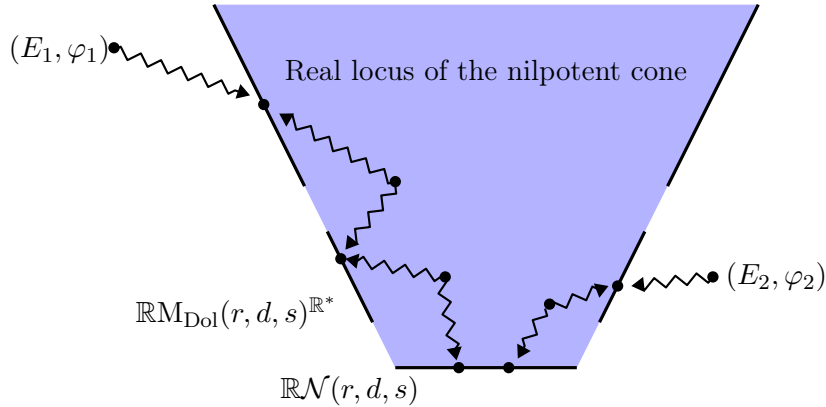
\begin{figure}[!h]
    \centering
    
\begin{tikzpicture}[scale=1.2]
    \fill[blue!30] (-1,0) -- (-3,4) to (3,4) -- (1,0) -- cycle;
    
    \draw[very thick] (-1.75,1.5) -- (-1.25,0.5);
    \draw[very thick] (-2,2) -- (-3,4);
    \draw[very thick] (2,2) -- (3,4);
    \draw[very thick] (1.75,1.5) -- (1.25,0.5);
    \draw[very thick] (-1,0) -- (1,0);
    
    \node at (-2.75,0.6) {$\mathbb{R}\mathrm{M}_{\mathrm{Dol}}(r,d,s)^{\mathbb{R}^*}$};
    \node at (0,3.25) {Real locus of the nilpotent cone};
    \node at (-1.5,-0.25) {$\mathbb{R}\mathcal{N}(r,d,s)$};
    
    \node[circle,fill,inner sep=1.5pt] at (-4.1,3.525) {};
    \node at (-4.7,3.5) {$(E_1,\phi_1)$};
            
    \node[circle,fill,inner sep=1.5pt] at (-2.45,2.9) {};
    
    \draw[-Triangle, thick, line join=round,
decorate, decoration={
    zigzag,
    segment length=8,
    amplitude=2,post=lineto,
    post length=2pt
}] 
        (-4.1,3.5) to (-2.6,3);
    
    \node[circle,fill,inner sep=1.5pt] at (-1,2.05) {};

    \node[circle,fill,inner sep=1.5pt] at (-1.6,1.2) {};

    \node[circle,fill,inner sep=1.5pt] at (-0.45,1) {};
    
    \draw[-Triangle, thick, line join=round,
decorate, decoration={
    zigzag,
    segment length=8,
    amplitude=2,post=lineto,
    post length=2pt
}] 
        (-1,2.05) to (-2.28,2.8);

    \node[circle,fill,inner sep=1.5pt] at (2.5,1) {};
            
    \node[circle,fill,inner sep=1.5pt] at (1.45,0.9) {};
    \node at (3.2,1) {$(E_2,\phi_2)$};
    
    \node[circle,fill,inner sep=1.5pt] at (-0.3,0) {};

    \draw[-Triangle, thick, line join=round,
decorate, decoration={
    zigzag,
    segment length=8,
    amplitude=2,post=lineto,
    post length=2pt
}] 
        (2.5,1) to (1.6,0.9);

    \draw[-Triangle, thick, line join=round,
decorate, decoration={
    zigzag,
    segment length=8,
    amplitude=2,post=lineto,
    post length=2pt
}] 
        (-1,2) to (-1.55,1.3);

    \draw[-Triangle, thick, line join=round,
decorate, decoration={
    zigzag,
    segment length=8,
    amplitude=2,post=lineto,
    post length=2pt
}] 
         (-0.5,1) to (-1.55,1.2);

    \draw[-Triangle, thick, line join=round,
decorate, decoration={
    zigzag,
    segment length=8,
    amplitude=2,post=lineto,
    post length=2pt
}] 
        (-0.5,1) to (-0.325,0.08);
        
    \node[circle,fill,inner sep=1.5pt] at (0.7,0.7) {};
    
    \draw[-Triangle, thick, line join=round,
decorate, decoration={
    zigzag,
    segment length=8,
    amplitude=2,post=lineto,
    post length=2pt
}] 
        (0.7,0.65) to (1.35,0.9) ;

    \node[circle,fill,inner sep=1.5pt] at (0.25,0) {};
    
    \draw[-Triangle, thick, line join=round,
decorate, decoration={
    zigzag,
    segment length=8,
    amplitude=2,post=lineto,
    post length=2pt
}] 
        (0.65,0.65) to (0.35,0.1);

\end{tikzpicture}

    \caption{Real Bialynicki-Birula flows in $\mathbb{R}\mathrm{M}_{\mathrm{Dol}}(r,d,s)$.}
    \label{fig_realBBflows}
\end{figure}

\begin{corollary}\label{CC_branes}
    The number of connected components of the real locus $\R\MD(r,d)$ of the Dolbeault moduli space is given, for arbitrary $r$ and $d$, by Table \ref{CC_count_curves_without_real_pts} when $\R X = \emptyset$ and by Theorem \ref{CC_count_curves_with_real_pts} when $\R X \neq \emptyset$. In particular, when $\R X$ has $n>0$ connected components, then $\R\MD(r,d)$ has $2^{n-1}$ connected components.
\end{corollary}

\subsection{Character varieties and branes in the Hitchin moduli space}

When $d=0$, the nonabelian Hodge correspondence establishes a homeomorphism $\MD(r,0) \simeq \MB(r,0)$ between the moduli space of semistable Higgs bundles of rank $r$ and degree $0$ and the so-called Betti moduli space $\MB(r,0) := \Hom\big(\pi_1 X ; \Gl(r,\C)\big) /\negmedspace/ \C$ which is an affine algebraic variety over $\C$. When 
$X$ has a real structure $\sigma$, there is an associated order $2$ element in $\mathrm{Out}(\pi_1 X)$, induced by the group morphism $\sigma_* : \pi_1(X, x) \to \pi_1 (X, \sigma(x))$. Equivalently, this order $2$ element of $\mathrm{Out}(\pi_1 X)$ is constructed from the short exact sequence $1 \to \pi_1 X \to \pi_1^{orb}(X/\left< \sigma \right>) \to \left< \sigma \right> \to 1$ by lifting the non-trivial element of $\left< \sigma \right> \simeq \Z/2\Z$ and conjugating by it on the normal subgroup $\pi_1 X$ of the orbifold fundamental group of $X/\left< \sigma \right>$. Since $\mathrm{Out}(\pi_1 X)$ acts on $\MB(r,0)$, we get an involutive automorphism of the Betti moduli space. Combining it with the canonical real structure $g \mapsto \ov{g}$ of $\Gl(r,\C)$, we get a real structure 
$$
\begin{array}{ccc}
    \beta : \MB(r,0) & \longrightarrow & \MB(r,0) \\
    \left[ \varrho \right] & \longmapsto & \left[ \ov{\varrho \circ \sigma_*} \right]
\end{array}
$$
which, through the nonabelian Hodge correspondence, is conjugate to the real structure \eqref{Dol_can_real_structure} of $\MD(r,0)$. The real locus of these anti-holomorphic involutions therefore defines an $(A,A,B)$-brane in the hyperkähler quotient $\mathrm{M}_\mathrm{Hit}(r,0)$, consisting of gauge equivalence classes of solutions to the Hitchin equations (\cite{ba2,BGP}). If we combine the order $2$ element of $\mathrm{Out}(\pi_1X)$ with the holomorphic (outer) involution $g \mapsto \ ^t g^{-1}$ of $\Gl(r,\C)$, we get an algebraic involution
$$
\begin{array}{ccc}
    \widetilde{\beta} : \MB(r,0) & \longrightarrow & \MB(r,0) \\
    \left[ \varrho \right] & \longmapsto & \left[ \ ^t\big(\varrho \circ \sigma_*\big)^{-1} \right]
\end{array}
$$
which, through the nonabelian Hodge correspondence, is conjugate to the real structure $(E, \varphi) \mapsto (\sigma^*(\ov{E}, - \sigma^*(\ov\phi))$ of $\MD(r,0)$ and therefore defines an $(A,B,A)$-brane in $\mathrm{M}_\mathrm{Hit}(r,0)$. We then obtain the following results as a consequence of Theorem \ref{theoremHiggsbundles}.

\begin{teorema}\label{CC_branes2}
    The fixed-point sets of $\beta$ and $\widetilde{\beta}$ are homeomorphic. If $\R X = \emptyset$, the number of connected components of these fixed-point sets is given by Table \ref{CC_count_curves_without_real_pts}. If $\R X$ has $n > 0$ connected components, $\mathrm{Fix}(\beta)$ and $\mathrm{Fix}(\widetilde{\beta})$ have $2^{n-1}$ connected components.
\end{teorema}

\begin{proof}
    It suffices to observe that, on the Dolbeault moduli space $\MD(r,0)$, the real structure $(E, \phi) \mapsto (\sigma^*(\ov{E}), \sigma^*(\ov{\phi}))$ and $(E, \phi) \mapsto (\sigma^*(\ov{E}), -\sigma^*(\ov{\phi}))$ are conjugate via the order four automorphism $(E, \phi) \mapsto (E, i \phi)$. The conclusion then follows from Corollary \ref{CC_branes}.
\end{proof}

The fixed-point sets $\mathrm{Fix}(\beta)$ and $\mathrm{Fix}(\widetilde{\beta})$ have attracted interest \cite{BHH, BGH} because they contain representations $\varrho: \pi_1 X \to \Gl(r,\C)$ that extend to the orbifold fundamental group $\pi_1^{orb}(X/\left<\sigma\right>)$ in a way that makes the following diagram commute.
$$
\begin{tikzcd}
    1 \ar[r] & \pi_1 X \ar[r] \ar[d] & \pi_1^{orb}(X/\left<\sigma\right>) \ar[r] \ar[d] & \left<\sigma\right> \ar[d] \ar[r] & 1  \\
    1 \ar[r] & \Gl(r,\C) \ar[r] & \Gl(r,\C) \rtimes \left<\sigma\right> \ar[r] & \left<\sigma\right> \ar[r] & 1
\end{tikzcd}
$$
where the action of $\left<\sigma\right>$ in the semi-direct product $\Gl(r,\C) \rtimes \left<\sigma\right>$ is given by $g \mapsto \ov{g}$ for the real structure $\beta$ and for $g \mapsto \ ^t g^{-1}$ for the algebraic involution $\widetilde{\beta}$. Note that we can define the map $\widehat{c_{r,0}} : \R \MD(r,0) \to \pi_0(\R\Pic_0)$ from \eqref{obstruction_map_Dolbeault} directly on $\R\MB(r,0)$, but we have to separate the cases when $\R X \neq \emptyset$ and $\R X = \emptyset$. Indeed, if the real locus $\R X$ consists of $n > 0$ circles $(\gamma_i)_{1 \leqslant i \leqslant n}$ and $[\varrho] \in \R\MB(r,0)$, then $\det \varrho(\gamma_i) \in \R^*$, and the associated sign $\pm 1$ corresponds to the Stiefel-Whitney classes of the real line bundle corresponding to $\det(\varrho) : \pi_1 X \to \C^*$ (see \cite[Section~4.3]{Sch_JDG} for details). Similarly, when $\R X = \emptyset$, we can choose a path between $x$ and $\sigma(x)$ to construct an element $\mathrm{sgn}(\det(\rho))\in\{\pm 1\}$ (see \cite{Sco_GD}). The resulting map $o_r : \R\MB(r,0) \to \pi_0(\R^*)$ is an obstruction map in the sense of Goldman \cite{Goldman_symp_nature} and Li \cite{JunLi_GroupReps} and we have proven that the fibres of $o_r$ are connected, giving the decomposition of $\MB(r,0)$ into connected components. For $d \neq 0$, we can replace $\pi_1 X$ by a central extension $\pi_d$ and obtain analogous statements. This was done in \cite[Theorems 4.2 and 4.5]{Sch-Sco} in the case when $\gcd(r,d) = 1$ and, as an application of Theorem \ref{theoremHiggsbundles}, these results can now be generalized to all $r$ and $d$.

\end{document}